\documentclass[12pt]{amsart}
\usepackage{a4}
\usepackage{amssymb}

\makeatletter
\@addtoreset{equation}{section}
\makeatother

\usepackage{multirow}
\usepackage{color}
\usepackage[mathscr]{eucal}
\usepackage{amsmath,amsthm,amssymb}
\usepackage{mathrsfs}
\usepackage{enumerate}
\usepackage{bm}
\usepackage{ulem} %\sout{ }で文章の上から横線を引いて消すことができる
\usepackage{graphicx}
\usepackage{verbatim}
\usepackage{wrapfig}
\usepackage{ascmac}
\usepackage{multicol}
\usepackage{latexsym}

\usepackage[dvipdfmx]{hyperref}

\newtheorem{Prop}{Proposition}[section]
\newtheorem{Thm}[Prop]{Theorem}

\newtheorem{Lem}[Prop]{Lemma}

\newtheorem{Rem}[Prop]{Remark}

\theoremstyle{definition}
\newtheorem{Def}[Prop]{Definition}

\if0

\newcommand{\R}{{\mathbb R}}

\newcommand{\Z}{{\mathbb Z}}
\newcommand{\GL}{\mathrm{GL}}
\newcommand{\SO}{\mathrm{SO}}
\newcommand{\OO}{\mathrm{O}}

\newcommand{\RH}{\R \mathrm{H}}

\newcommand{\id}{\mathrm{id}}
\newcommand{\Aut}{\mathrm{Aut}}

\newcommand{\rnum}[1]{\expandafter{\romannumeral #1}}

\fi
\newcommand{\trans}{{{}^t\!}}

\newcommand{\R}{{\mathbb R}}

\newcommand{\Z}{{\mathbb Z}}

\newcommand{\LG}{{\mathfrak g}}
\newcommand{\LH}{{\mathfrak h}}

\newcommand{\Oo}{\mathcal{O}}
\newcommand{\Dd}{\mathcal{D}}
\newcommand{\LLM}{{\mathfrak M}}

\newcommand{\LLU}{{\mathfrak U}}
\newcommand{\llu}{{\mathfrak u}}

\newcommand{\Ric}{\mathrm{Ric}}
\newcommand{\Der}{\mathrm{Der}}
\newcommand{\Aut}{\mathrm{Aut}}

\newcommand{\id}{\mathrm{id}}

\newcommand{\GL}{\mathrm{GL}}
\newcommand{\SO}{\mathrm{SO}}
\newcommand{\OO}{\mathrm{O}}

\newcommand{\s}{\mathrm{span}}
\newcommand{\codim}{\mathrm{codim}}
\newcommand{\gl}{\mathfrak{gl}}

\newcommand{\oo}{\mathfrak{o}}
\newcommand{\RH}{\R \mathrm{H}}

\newcommand{\RAutg}{\R^\times \Aut(\LG)}

\def\hoge<#1>{\langle #1 \rangle} %内積

\title[A classification of left-invariant Lorentzian metrics]{A classification of left-invariant Lorentzian metrics on some nilpotent Lie groups} 

\author{Yuji Kondo} 
\address[Y.~Kondo]{Department of Mathematics, Hiroshima University, 
Higashi-Hiroshima, Japan 739-8526} 
\email{yuji-kondo@hiroshima-u.ac.jp}

\author{Hiroshi Tamaru}
\address[H.~Tamaru]{Department of Mathematics, Osaka City University, Osaka, Japan 558-8585} 
\email{tamaru@sci.osaka-cu.ac.jp}

\thanks{This work was partly supported by Osaka City University Advanced Mathematical Institute (MEXT Joint Usage/Research Center on Mathematics and Theoretical Physics). The second author was supported by JSPS KAKENHI Grant Number JP19K21831.} 

\date{}



\begin{document} 

\maketitle

\begin{abstract} 
It has been known that there exist exactly three left-invariant Lorentzian metrics up to scaling and automorphisms on the three dimensional Heisenberg group. In this paper, we classify left-invariant Lorentzian metrics on the direct product of three dimensional Heisenberg group and the Euclidean space of dimension $n-3$ with $n \geq 4$, and prove that there exist exactly six such metrics on this Lie group up to scaling and automorphisms. Moreover we show that only one of them is flat, and the other five metrics are Ricci solitons but not Einstein. We also characterize this flat metric as the unique closed orbit, where the equivalence class of each left-invariant metric can be identified with an orbit of a certain group action on some symmetric space.
\end{abstract}

%\tableofcontents %目次を出力

\section{Introduction}
\label{sec1} 

Left-invariant metrics on Lie groups, both in Riemannian and pseudo-Riemannian cases, have been studied actively. Among others, classifications of left-invariant metrics are fundamental and interesting themes. For example, Milnor classified left-invariant Riemannian metrics on three dimensional unimodular Lie groups by using orthonormal bases of Lie algebras in \cite{Milnor}, which are now called the {\it Milnor frames}. Note that the Milnor frames play fundamental roles in studying Ricci soliton metrics (cf. \cite{Topping}). In general, if we can classify left-invariant metrics on a given Lie group, then it would be helpful to determine the existence and nonexistence of distinguished metrics, such as Einstein or Ricci soliton, which is one of the central problems.
%Especially, Einstein metrics and Ricci soliton metrics on Lie groups are intriguing topics. 
%(\cite{Besse, Heber, Lauret1, Lauret2}). 
%One of the central problems is to study whether a given Lie group can admit these kinds of distinguished metrics. 
%In the Riemannian cases, there are some researches about the existence and nonexistence of distinguished metrics. But in the pseudo-Riemannian cases, the number of known results is less than the Riemannian cases. \textcolor{red}{(存在非存在についてリーマンの場合にはいろいろ研究があるが、擬リーマンの場合にはあまりない. 一応書いた)}
%\textcolor{red}{(左不変計量の分類の話をメインにしたい。分類ができればアインシュタインやリッチソリトンもわかるという感じ。ミルナーを起に書く)}

In the Riemannian case, Lauret (\cite{Lauret}) classified Lie groups admitting only one left-invariant Riemannian metric up to scaling and isometry. Such a Lie group is isomorphic to, if it is connected and simply-connected, one of 
\begin{align*}
\R^n, \quad G_{\RH^n} \ (n \geq 2), \quad H_3 \times \R^{n-3} \ (n \geq 3),
\end{align*}
where $G_{\RH^n}$ is so-called the Lie group of the real hyperbolic space $\RH^n$ (the solvable part of the Iwasawa decomposition of the identity component $\SO^0(n, 1)$ of $\SO(n, 1)$ and acts simply-transitively on $\RH^n$), and $H_3$ is the three dimensional Heisenberg group. For other studies on classifications of left-invariant Riemannian metrics on Lie groups, 
%\textcolor{red}{(\cite{HTT, KTT}の話。３次元可解の場合のミルナー型定理は\cite{HT}で構成されている)} 
we refer to \cite{HT, HTT, KTT}. Especially, in \cite{HTT}, a kind of theorem to classify left-invariant Riemannian metrics on Lie groups is formulated, which is called a {\it Milnor-type theorem}. In \cite{HT}, Milnor-type theorems have been obtained for left-invariant Riemannian metrics on all three dimensional solvable Lie groups. However, even in the Riemannian case, the present status is far from the completion. For example, left-invariant Ricci soliton metrics on solvable Lie groups have been classified only for dimension $\leq 6$ (\cite{Lauret2, Will}).
%\textcolor{red}{(他にもいろいろ研究はあるけど、リーマンの場合にも完全にわかっているわけではない. 一応書いた)}

%\textcolor{red}{(承の大まかな流れ。Lauret や HT, KTT, HTT など左不変リーマン計量の分類に関する研究はいろいろある。しかし完全にわかっているわけではない。例えば solvsoliton やアインシュタインの分類は低次元の場合にしかまだわかっていない。)}

We are interested in classifications of left-invariant pseudo-Riemannian metrics on Lie groups. Left-invariant Lorentzian metrics on three dimensional Lie groups have been studied in \cite{CP, Rahmani, RR}.
%\textcolor{red}{In the three dimensional cases, the curvature tensors of left-invariant Lorentzian metrics have been calculated by using Milnor frame (\cite{CP}).} For the case of $H_3$, it admits exactly three left-invariant Lorentzian metrics up to scaling and automorphisms (\cite{Rahmani, RR}), and only one of them is flat and the other two are Ricci solitons but not Einstein (\cite{Nomizu, Onda, Onda1, RR}). 
For higher dimensional cases, it would be natural to start with the above three Lie groups, that is $\R^n$, $G_{\RH^n}$, and $H_3 \times \R^{n-3}$. For each signature, it is obvious that $\R^n$ admits only one left-invariant pseudo-Riemannian metric up to scaling and isometry, which is flat. For each non-Riemannian signature on $G_{\RH^n}$ $(n \geq 2)$, it admits exactly three left-invariant pseudo-Riemannian metrics up to scaling and isometry, all of which have constant sectional curvatures (\cite{KOTT}). For the case of $H_3$, it admits exactly three left-invariant Lorentzian metrics (\cite{Rahmani}), only one of which is flat and the other two are Ricci solitons but not Einstein (\cite{Nomizu, Onda, Onda1, RR}). However, the case of $H_3 \times \R^{n-3}$ with $n \geq 4$ is unsolved. %Of course, also in the pseudo-Riemannian cases, left-invariant metrics on Lie groups are interesting and have been studied. For example, every left-invariant Lorentzian metric on Lie groups in a certain class has a constant sectional curvature and any given real number can be realized as the value of its constant sectional curvature (\cite{Nomizu}). 
%But compared to the Riemannian cases, the pseudo-Riemannian cases are far from completion.

In this paper, we give a classification of left-invariant Lorentzian metrics on $H_3 \times \R^{n-3}$ with $n \geq 4$, up to scaling and automorphisms. Recall that this criterion of classification is defined as follows. 

\begin{Def}
\label{def : up to scaling and auto}
Let $g_1$ and $g_2$ be left-invariant pseudo-Riemannian metrics on a Lie group $G$. Then, $(G, g_1)$ and $(G, g_2)$ are said to be {\it equivalent up to scaling and automorphisms} if there exist $c>0$ and a Lie group automorphism $\varphi: G \to G$ such that for any $p \in G$ and $x, y \in T_pG$, it satisfies
\begin{align*}
g_1(x, y)_p=cg_2(d\varphi_p(x), d\varphi_p(y))_{\varphi(p)},
\end{align*}
where $T_pG$ is the tangent space to $p$ of $G$, and $d\varphi_p$ is the differential map of $\varphi$ at $p$.
\end{Def}

If $(G, g_1)$ and $(G, g_2)$ are equivalent up to scaling and automorphisms, then they are isometric up to scaling. Note that the converse is not necessarily true. The first main result of this paper classifies left-invariant Lorentzian metrics on $H_3 \times \R^{n-3}$ up to scaling and automorphisms. We have to note that this does not give a classification up to scaling and isometry (see Remark~\ref{rem : isometric up to scaling on Lie group}).

\begin{Thm}
\label{thm : classification of inner products}
There exist exactly six left-invariant Lorentzian metrics on $H_3 \times \R^{n-3}$ $(n \geq 4)$ up to scaling and automorphisms.
\end{Thm}

In the proof of this theorem, a key idea is a one-to-one correspondence between the equivalence classes of left-invariant Lorentzian metrics on $H_3 \times \R^{n-3}$ up to scaling and automorphisms, and orbits of some group action. In fact, this group action is given by the non-maximal parabolic subgroup
\begin{align*}
\left\{\left(
		\begin{array}{cc|ccc|c}
		\ast&\ast&0&\cdots&0&0\\ 
		\ast&\ast&0&\cdots&0&0\\ \hline
		\ast&\ast&\ast&\cdots&\ast&0\\
		\vdots&\vdots&\vdots&\ddots&\vdots&\vdots\\ 
		\ast&\ast&\ast&\cdots&\ast&0\\ \hline
		\ast&\ast&\ast&\cdots&\ast&\ast
		\end{array}
	\right) \in \GL(n, \R)\right\}
\end{align*}
in $\GL(n, \R)$, where the size of the block decomposition is $(2, n-3, 1)$, acting on the pseudo-Riemannian symmetric space $\GL(n, \R)/\OO(n-1, 1)$. Note that it has already been known that the number of orbits is finite for this action.  Our argument asserts that the number of orbits of this action is exactly six (see Remark~\ref{rem : wolf}). This result would have an independent interest. Recall that, in the cases of $G_{\RH^n}$ $(n \geq 2)$ and $H_3$, the corresponding actions are given by maximal parabolic subgroups in $\GL(n, \R)$, and there are exactly three orbits (\cite{KOTT, Rahmani}). In our case, since one has to study the action of a smaller group, we need more detailed arguments and the number of orbits increases.

The second main result of this paper studies the curvature properties of the above metrics. In fact, we obtain a Milnor-type theorem for the Lie group $H_3 \times \R^{n-3}$, which gives a kind of generalization of Milnor frames. By calculating the curvatures in terms of the obtained Milnor-type theorem, we prove the following.

\begin{Thm}
\label{thm : Ricci and Einstein}
All of the six left-invariant Lorentzian metrics obtained in Theorem~\ref{thm : classification of inner products} are Ricci soliton metrics. Only one of them is flat, and the other five are not Einstein. 
\end{Thm}

We also study the closure relation among these six orbits. In view of the correspondence between the orbits and the equivalence classes of the metrics mentioned above, it would be natural to expect that some distinguished orbits are corresponding to some distinguished metrics. This expectation turned out to be true in our case, which is the last main result. Note that closed orbits do not degenerate further, and hence can be regarded as the most distinguished orbits from the viewpoint of the degenerations.

\begin{Thm}
\label{thm : flat and closed}
A left-invariant Lorentzian metric on $H_3 \times \R^{n-3}$ $(n \geq 4)$ is flat if and only if the corresponding orbit is a closed orbit, that is, its equivalence class up to scaling and automorphisms is a closed set in $\GL(n,\R)/\OO(n-1, 1)$.
\end{Thm}

In the preceding studies, for $G_{\RH^n}$ $(n \geq 2)$ and $H_3$, one has the same correspondences. In fact, for each non-Riemannian signature, these Lie groups admit unique flat left-invariant metrics up to scaling and automorphisms (\cite{KOTT}, \cite{RR}). One can also see that these flat left-invariant pseudo-Riemannian metrics are exactly corresponding to the unique closed orbits. It would be interesting to study whether this kind of nice correspondences also hold for other cases. 

%Note that we already know that $H_3$ admits a flat left-invariant Lorentzian metric. We mention that three among these five non-Einstein Ricci solitons are new examples.

%At last we mention something important. We say that left-invariant Lorentzian metrics on $H_3$ have been classified up to isometry in \cite{RR}. But to be honest, we are not quite sure that their result is correct or not. First, S. Rahmani proved that there was exactly three left-invariant Lorentzian metrics on $H_3$ up to automorphism in \cite{Rahmani}. And next, in Theorem 2 of \cite{RR} they claimed that the classification obtained in \cite{Rahmani} was equivalent to the classification up to isometry. But in order to prove this theorem, they used Theorem 3 in \cite{Wilson} whose proof seemed to be strongly depend on positive definiteness of Riemannian metrics. We are interested in the pseudo-Riemannian cases, not necessarily positive definite, that is why we wonder if the result in \cite{RR} is correct or not.

The authors would like to thank Takayuki Okuda, Akira Kubo and Yuichiro Taketomi for valuable comments and suggestions. The authors are also grateful to Toshihiko Matsuki for precious advice, which have a strong influence to our studies.

\section{Preliminaries}
\label{sec2}

In this section, we recall a general theory on left-invariant metrics on Lie groups, both for Riemannian and pseudo-Riemannian. Throughout this section, let $G$ be a real Lie group of dimension $n$ and $\LG$ be its corresponding Lie algebra. We fix a basis $\{e_1, \ldots, e_n\}$ of $\LG$, and identify $\LG \cong \R^n$ as vector spaces.

\subsection{The spaces of left-invariant metrics on Lie groups}

In this subsection, we recall the notion of the spaces of left-invariant pseudo-Riemannian metrics on Lie groups. This notion has been introduced in \cite{KOTT}. We also refer to \cite{KTT} for the Riemannian case.

First of all, let us recall the signature of an inner product. Let $V$ be a real vector space of dimension $n$ and $\hoge< , >$ be an inner product on $V$, which is not necessarily positive definite. Fix a basis $\{v_1, \ldots, v_n\}$ of $V$ and identify $V \cong \R^n$. Then, there exists a symmetric matrix $A$ such that for any $x, y \in V$, 
\begin{align*}
\hoge<x, y>=\trans xAy.
\end{align*}
Then the pair of the numbers of positive and negative eigenvalues of $A$ is called the {\it signature} of $\hoge< , >$. Note that the signature $(p, q)$ with $p, q \in \Z_{\geq 0}$ of $\hoge< , >$ satisfies $p+q=n$, since $\hoge< , >$ is nondegenerate.

Next we consider left-invariant pseudo-Riemannian metrics on $G$. Recall that a metric is said to be of signature $(p, q)$ if so is the induced inner product on each tangent space. We are interested in a classification of left-invariant pseudo-Riemannian metrics on $G$. For this purpose, we  denote the space of left-invariant pseudo-Riemannian metrics by
\begin{align*}
\LLM_{(p, q)}(G):=\{\mbox{a\ left-invariant\ metric\ of\ signature}\ (p, q)\ \mbox{on\ }G\}.
\end{align*}

We then consider the counterpart in the Lie algebra $\LG$ of $G$. It is well-known that there is a one-to-one correspondence between $\LLM_{(p, q)}(G)$ and the space of inner products of the same signature,
\begin{align*}
\LLM_{(p, q)}(\LG):=\{\hoge< , > : \mbox{an inner product of signature}\ (p, q)\ \mbox{on}\ \LG\}.
\end{align*}
Recall that we identify $\LG \cong \R^n$. Then $\GL(n, \R)$ acts transitively on this space by
\begin{align*}
g.\hoge<x, y>:=\hoge<g^{-1}x, g^{-1}y> \quad (\forall x, y \in \LG).
\end{align*}

From now on, we explain the equivalence relation on inner products, which corresponds to the equivalence relation on $\LLM_{(p, q)}(G)$ given by Definition~\ref{def : up to scaling and auto}.
%The space that we consider in this paper is orbit spaces of a particular group action on $\LLM_{(p, q)}(\LG)$. 
Let us consider the automorphism group of $\LG$, 
\begin{align*}
\Aut(\LG):=\{\varphi \in \GL(n, \R) \mid \forall x, y \in \LG, \varphi([x, y])=[\varphi(x), \varphi(y)]\}.
\end{align*}
We also put $\R^\times:=\R \setminus \{0\}$. In this paper we consider the group action by 
\begin{align*}
\R^\times \Aut(\LG):=\{c\varphi \in \GL(n, \R) \mid c \in \R^\times, \varphi \in \Aut(\LG)\}.
\end{align*}
Since this is a subgroup of $\GL(n, \R)$, it naturally acts on $\LLM_{(p, q)}(\LG)$. We denote the orbit through $\hoge< , >$ by $\R^\times \Aut(\LG).\hoge< , >$.

\begin{Def}
\label{def : isometric up to scaling}
Let $\hoge< , >_1, \hoge< , >_2 \in \LLM_{(p, q)}(\LG)$. Then, $(\LG, \hoge< , >_1)$ and $(\LG, \hoge< , >_2)$ are said to be {\it equivalent up to scaling and automorphisms} if it satisfies
\begin{align*}
\hoge< , >_1 \in \R^\times \Aut(\LG).\hoge< , >_2.
\end{align*}
\end{Def}

This notion gives an equivalence relation on $\LLM_{(p, q)}(\LG)$. If a given Lie group $G$ is connected and simply-connected, then one knows $\Aut(G) \cong \Aut(\LG)$, and hence the classification of inner products on $\LG$ by the action of $\RAutg$ is equivalent to the classification of left-invariant pseudo-Riemannian metrics on $G$ up to scaling and automorphisms. Therefore it is natural to consider the following orbit space:
\begin{align*}
\R^\times \Aut(\LG) \backslash \LLM_{(p, q)}(\LG):=\{\R^\times \Aut(\LG).\hoge< , > \mid \hoge< , > \in \LLM_{(p, q)}(\LG)\}.
\end{align*}
This space can be regarded as the moduli space of left-invariant pseudo-Riemannian metrics on $G$ of signature $(p, q)$.

Finally in this subsection, we give a remark on a classification of left-invariant pseudo-Riemannian metrics on $G$ up to scaling and isometry, defined as follows.

\begin{Def}
\label{def : isometric up to scaling on Lie group}
Let $g_1, g_2 \in \LLM_{(p, q)}(G)$. Then, $(G, g_1)$ and $(G, g_2)$ are said to be {\it isometric up to scaling} and denoted by $g_1 \sim_G g_2$ if there exist $c>0$ and a diffeomorphism $\varphi: G \to G$ such that for any $p \in G$ and $x, y \in T_pG$,
\begin{align*}
g_1(x, y)_p=cg_2(d\varphi_p(x), d\varphi_p(y))_{\varphi(p)}.
\end{align*}
\end{Def}

One can define an equivalence relation $\sim_{\LG}$ on $\LLM_{(p, q)}(\LG)$ induced from $\sim_G$, that is, there exists a one-to-one correspondence
\begin{align*}
\LLM_{(p, q)}(G)/\sim_G \overset{1:1}{\longleftrightarrow} \LLM_{(p, q)}(\LG)/\sim_\LG.
\end{align*}
By definition, if two left-invariant metrics are equivalent up to scaling and automorphisms, then they are isometric up to scaling. Therefore there exists a surjection
\begin{align*}
\RAutg \backslash \LLM_{(p, q)}(\LG) \twoheadrightarrow \LLM_{(p, q)}(\LG)/\sim_\LG.
\end{align*}
In this paper, as we referred above, we focus on the classification of inner products by the action of $\RAutg$. In order to obtain the classification up to $\sim_G$ or $\sim_{\LG}$, we need to distinguish elements in $\RAutg \backslash \LLM_{(p, q)}(\LG)$, which can be equivalent in the sense of $\sim_\LG$.

\subsection{A set of representatives}

As in the previous section, a classification of left-invariant pseudo-Riemannian metrics on $G$ up to scaling and automorphisms is equivalent to determine the orbit space $\RAutg \backslash \LLM_{(p, q)}(\LG)$. In order to determine the orbit space, the notion of set of representatives is useful. In this subsection, we recall this notion and some of the properties. We also recall the procedure to obtain Milnor-type theorems.

Let $I_k$ be the unit matrix of order $k$, and put
\begin{align*}
I_{p, q}:=\left(
		\begin{array}{cc}
		I_p&\\
		&-I_q
		\end{array}
		\right).
\end{align*}
We consider the canonical inner product of signature $(p, q)$ on $\LG \cong \R^{p+q}$, defined by
\begin{align*}
\hoge<x, y>_0:=\trans xI_{p, q}y \quad (\forall x, y \in \LG).
\end{align*}

\begin{Def}
Let $H$ be a subgroup of $\GL(p+q, \R)$ and consider the action of $H$ on $\LLM_{(p, q)}(\LG)$. Then, a subset $\LLU \subset \GL(p+q, \R)$ is called a {\it set of representatives} of this action if the orbit space satisfies
\begin{align*}
H \backslash \LLM_{(p, q)}(\LG)=\{H.(g_0.\hoge< , >_0) \mid g_0 \in \LLU\}.
\end{align*}
\end{Def}

In order to obtain a set of representatives $\LLU$, the notion of double cosets is useful. Recall that the indefinite orthogonal group $\OO(p, q)$ is defined as the isotropy subgroup of $\GL(p+q, \R)$ at $\hoge< , >_0$. One thus has an expression as homogeneous space
\begin{align*}
\LLM_{(p, q)}(\LG)=\GL(p+q, \R)/\OO(p, q),
\end{align*}
by which one can see that $\LLM_{(p, q)}(\LG)$, and hence $\LLM_{(p, q)}(G)$, is a pseudo-Riemannian symmetric space. It also follows that the orbit space $H \backslash \LLM_{(p, q)}(\LG)$ can be represented as a double coset space. Then one knows the following by a standard theory of double coset spaces.

\begin{Lem}[cf.~\cite{KOTT}]
\label{lem : double coset}
Consider an action of a subgroup $H \subset \GL(p+q, \R)$ on $\LLM_{(p, q)}(\LG)$. Then, a subset $\LLU \subset \GL(p+q, \R)$ is a set of representatives of this action if and only if for any $g \in \GL(p+q, \R)$, there exists $g_0 \in \LLU$ such that $g_0 \in Hg\OO(p, q)$.
\end{Lem}

Next we describe two lemmas, which we use to calculate a set of representatives for our case in Section~\ref{sec3}. The first one is about an action of $\OO(1, 1)$. Note that $\OO(1, 1)$ is naturally a subgroup of $\OO(p, q)$ for $p, q \geq 1$.

\begin{Lem}[\cite{KOTT}]
\label{lem : hyperbola}
Let $(x, y) \neq (0, 0)$. Then, there exist $a>0$, $\lambda \in \{0, 1, 2\}$, and $g \in \OO(1, 1)$ such that $(x, y)g=(-\lambda a, a)$ holds.
\end{Lem}

The second lemma states the correspondence between sets of representatives of the actions of $H$ and $H'$, where the latter group is defined by
\begin{align*}
H':=\{\trans h \mid h \in H\}.
\end{align*}

\begin{Lem}[\cite{KOTT}]
\label{lem : action of transpose}
Let $H$ be a subgroup of $\GL(p+q, \R)$, and $\LLU$ be a set of representatives of the action of $H$ on $\LLM_{(p, q)}(\LG)$. Then, the following $\LLU^\ast$ is a set of representatives of the action of $H'$ on $\LLM_{(p, q)}(\LG):$
\begin{align*}
\LLU^\ast:=\{\trans u^{-1} \mid u \in \LLU\}.
\end{align*}
\end{Lem}

Finally in this subsection, we describe a theorem which gives a procedure to obtain Milnor-type theorems. For this purpose, we need the notion of pseudo-orthonormal bases. We put
\begin{align*}
\varepsilon_i:=\left\{
			\begin{array}{ll}
				1&(i \in \{1, \ldots, p\}),\\
				-1&(i \in \{p+1, \ldots, p+q\}).
			\end{array}
			\right.
\end{align*}
We also use the Kronecker's delta $\delta_{ij}$. Then, a basis $\{x_1, \ldots, x_{p+q}\}$ of $\LG$ is said to be {\it pseudo-orthonormal} with respect to $\hoge< , > \in \LLM_{(p, q)}(\LG)$ if it satisfies
\begin{align*}
\hoge<x_i, x_j>=\varepsilon_i\delta_{ij} \quad (\forall i, j \in \{1, \ldots, p+q\}).
\end{align*}

\begin{Thm}[\cite{KOTT}]
\label{thm : procedure}
Let $\LLU$ be a set of representatives of the action of $\R^\times \Aut(\LG)$ on $\LLM_{(p, q)}(\LG)$. Then, for every inner product $\hoge< , >$ of signature $(p, q)$ on $\LG$, there exist $k>0$, $\varphi \in \Aut(\LG)$ and $g_0 \in \LLU$ such that $\{\varphi g_0e_1, \ldots, \varphi g_0e_{p+q}\}$ is pseudo-orthonormal with respect to $k\hoge< , >$.
\end{Thm}

If we know an expression of a set of representatives $\LLU$, then we can apply this theorem to a given Lie algebra, and obtain a pseudo-orthonormal basis. One can study properties of the inner product, such as the equivalence problem and curvature properties, in terms of this basis.

\section{Calculations of a set of representatives} 
\label{sec3} 

From now on, we consider left-invariant Lorentzian metrics on the Lie group $G:=H_3 \times \R^{n-3}$ with $n \geq 4$. For this purpose, we study its corresponding Lie algebra
\begin{align*}
\LG:=\LH_3 \oplus \R^{n-3}:=\s\{e_1, \ldots, e_n \mid [e_1, e_2]=e_n\},
\end{align*} 
where $\LH_3=\s\{e_1, e_2, e_n\}$ is the three dimensional Heisenberg Lie algebra. In this section, we calculate a set of representatives of the action of $\RAutg$ on $\LLM_{(n-1, 1)}(\LG)$. First of all, we recall a matrix expression of $\RAutg$.

\begin{Prop}[\cite{KTT}]
\label{prop : aut}
The matrix expression of $\RAutg$ with respect to a basis $\{e_1, \ldots, e_n\}$ of $\LG$ coincides with
\begin{align*}
\RAutg=\left\{\left(
		\begin{array}{cc|ccc|c}
		\ast&\ast&0&\cdots&0&0\\ 
		\ast&\ast&0&\cdots&0&0\\ \hline
		\ast&\ast&\ast&\cdots&\ast&0\\
		\vdots&\vdots&\vdots&\ddots&\vdots&\vdots\\ 
		\ast&\ast&\ast&\cdots&\ast&0\\ \hline
		\ast&\ast&\ast&\cdots&\ast&\ast
		\end{array}
	\right) \in \GL(n, \R)\right\}.
\end{align*}
\end{Prop}

\begin{Rem}
\label{rem : wolf}
In \cite{Wolf}, Wolf has obtained the finiteness of orbits of actions of parabolic subgroups on symmetric spaces of reductive type. Note that $\LLM_{(n-1, 1)}(\LG)$ is a pseudo-Riemannian symmetric space of reductive type, and the group $\RAutg$ is a parabolic subgroup of $\GL(n, \R)$ for $\LG:=\LH_3 \oplus \R^{n-3}$. It then follows that the number of orbits of this action is finite. The result of this section yields that there are at most six orbits. In Section~\ref{sec4}, we will show that the number of orbits is exactly six.
\end{Rem}

In order to make calculations slightly easier, let us consider the action of 
\begin{align*}
H'&:=\{\trans h \mid h \in \RAutg\} \nonumber\\
&=\left\{\left(
		\begin{array}{cc|ccc|c}
		\ast&\ast&\ast&\cdots&\ast&\ast\\ 
		\ast&\ast&\ast&\cdots&\ast&\ast\\ \hline
		0&0&\ast&\cdots&\ast&\ast\\
		\vdots&\vdots&\vdots&\ddots&\vdots&\vdots\\ 
		0&0&\ast&\cdots&\ast&\ast\\ \hline
		0&0&0&\cdots&0&\ast
		\end{array}
	\right) \in \GL(n, \R)\right\}.
\end{align*}
We will give a set of representatives of the action of $H'$ on $\LLM_{(n-1, 1)}(\LG)$. According to Lemma~\ref{lem : double coset}, we need to study the double cosets
\begin{align*}
[[g]]:=H'g \OO(n-1, 1).
\end{align*}
First of all, we divide the double cosets into three types.

\begin{Lem}
\label{lem : lambda=0, 1, 2}
Let $g \in \GL(n, \R)$. Then, there exists $\lambda \in \{0, 1, 2\}$ such that
\begin{align*}
\left(\begin{array}{cccc|c}
			\ast&\cdots&\cdots&\ast&0\\
			\vdots&\ddots&&\vdots&\vdots\\
			\vdots&&\ddots&\vdots&\vdots\\
			\ast&\cdots&\cdots&\ast&0\\ \hline
			-\lambda&0&\cdots&0&1
			\end{array}\right) \in [[g]].
\end{align*}
\end{Lem}

\begin{proof}
The proof of this lemma is similar to the arguments in \cite{KOTT}. Take an arbitrary $g \in \GL(n, \R)$. First of all, one knows that there exists $\alpha \in \OO(n-1)$ such that 
\begin{align*}
[[g]] \ni g\left(\begin{array}{ccc|c}
			&&&0\\
			&\alpha&&\vdots\\
			&&&0\\ \hline
			0&\cdots&0&1
			\end{array}\right)=
		\left(\begin{array}{cccc|c}
			\ast&\cdots&\cdots&\ast&\ast\\
			\vdots&\ddots&&\vdots&\vdots\\
			\vdots&&\ddots&\vdots&\vdots\\
			\ast&\cdots&\cdots&\ast&\ast\\ \hline
			x&0&\cdots&0&y
			\end{array}\right)=:g_1.
\end{align*}
Here, note that $(x, y) \neq (0, 0)$ since $\det g_1 \neq 0$. Hence, by Lemma~\ref{lem : hyperbola}, one can change $(x, y)$ into a certain form by $\OO(1, 1)$. Since this $\OO(1, 1)$ can be seen naturally as a subgroup of $\OO(n-1, 1)$, we obtain that there exist $a>0$, $\lambda \in \{0, 1, 2\}$ and $k_1 \in \OO(n-1, 1)$ such that 
\begin{align*}
[[g]] \ni g_1k_1=\left(\begin{array}{cccc|c}
			\ast&\cdots&\cdots&\ast&a_n\\
			\vdots&\ddots&&\vdots&\vdots\\
			\vdots&&\ddots&\vdots&\vdots\\
			\ast&\cdots&\cdots&\ast&a_2\\ \hline
			-\lambda a&0&\cdots&0&a
			\end{array}\right)=:g_2,
\end{align*}
where $a_2, \ldots, a_n \in \R$. Since $a>0$, it follows from the definition of $H'$ that
\begin{align*}
[[g]] \ni \left(\begin{array}{cccc|c}
			a&&&0&-a_n\\
			&\ddots&&&\vdots\\
			&&\ddots&&\vdots\\
			0&&&a&-a_2\\ \hline
			0&\cdots&\cdots&0&1/a
			\end{array}\right)g_2=
	\left(\begin{array}{cccc|c}
			\ast&\cdots&\cdots&\ast&0\\
			\vdots&\ddots&&\vdots&\vdots\\
			\vdots&&\ddots&\vdots&\vdots\\
			\ast&\cdots&\cdots&\ast&0\\ \hline
			-\lambda&0&\cdots&0&1
			\end{array}\right),
\end{align*}
which completes the proof.
\end{proof}

According to this lemma, we define the subsets of $\GL(n, \R)$ as follows :
\begin{align*}
%\label{definition of G_lambda}
G_\lambda:=\left\{\left(\begin{array}{cccc|c}
			\ast&\cdots&\cdots&\ast&0\\
			\vdots&\ddots&&\vdots&\vdots\\
			\vdots&&\ddots&\vdots&\vdots\\
			\ast&\cdots&\cdots&\ast&0\\ \hline
			-\lambda&0&\cdots&0&1
			\end{array}\right) \in \GL(n, \R)\right\}.
\end{align*}
Then we have only to study the double cosets $[[g]]$ where $g \in G_\lambda$ with $\lambda \in \{0, 1, 2\}$. The first case is $\lambda=0$. For this case, every $g \in G_0$ gives the same double coset.

\begin{Prop}
\label{lem : lambda=0}
For every $g \in G_0$, we have $I_n \in [[g]]$.
\end{Prop}

\begin{proof}
Take any $g \in G_0$, and denote it as
\begin{align*}
g=\left(\begin{array}{ccc|c}
			&&&0\\
			&\alpha_1&&\vdots\\
			&&&0\\ \hline
			0&\cdots&0&1
			\end{array}\right),
\end{align*}
where $\alpha_1 \in \GL(n-1, \R)$. Then there exists $\alpha_2 \in \OO(n-1)$ such that $\alpha_1\alpha_2$ is upper triangular. We note that
\begin{align*}
[[g]] \ni g\left(\begin{array}{ccc|c}
			&&&0\\
			&\alpha_2&&\vdots\\
			&&&0\\ \hline
			0&\cdots&0&1
			\end{array}\right)=
\left(\begin{array}{ccc|c}
			&&&0\\
			&\alpha_1\alpha_2&&\vdots\\
			&&&0\\ \hline
			0&\cdots&0&1
			\end{array}\right)=:g_1.
\end{align*}
By definition, one knows $g_1 \in H'$. Since $g_1^{-1} \in H'$, we obtain
\begin{align*}
[[g]] \ni g_1^{-1}g_1=I_n,
\end{align*}
which completes the proof.
\end{proof}

Before we consider the remaining cases $\lambda=1, 2$, we prove the next lemma. It will be used for the both cases. 

\begin{Lem}
\label{lem : procedure of lambda=1, 2}
Let $g \in G_\lambda$. Then there exists $t \in \R$ such that
\begin{align*}
\left(\begin{array}{c|cccccc}
	1&0&\cdots&\cdots&\cdots&\cdots&0\\ \hline
	0&&&&&&\\
	t&&&&&&\\
	0&&&\multicolumn{2}{c}{\multirow{2}{*}{$I_{n-1}$}}&&\\
	\vdots&&&&&&\\
	0&&&&&&\\
	-\lambda&&&&&&
	\end{array}\right) \in [[g]]. 
\end{align*}
\end{Lem}

\begin{proof}
Take any $g \in G_\lambda$. First of all, we convert the first column vector of $g$. There exists $\alpha_1 \in \OO(n-3)$ such that 
\begin{align*}
[[g]] \ni 
\left(\begin{array}{cc|ccc|c}
	1&0&0&\cdots&0&0\\
	0&1&0&\cdots&0&0\\ \hline
	0&0&&&&0\\
	\vdots&\vdots&&\alpha_1&&\vdots\\
	0&0&&&&0\\ \hline
	0&0&0&\cdots&0&1\\
	\end{array}\right)g=
\left(\begin{array}{cc|cccc|c}
	\ast&\ast&\ast&\cdots&\cdots&\ast&0\\
	\ast&\ast&\ast&\cdots&\cdots&\ast&0\\ \hline
	\ast&\ast&\ast&\cdots&\cdots&\ast&0\\
	0&\ast&\vdots&\ddots&&\vdots&\vdots\\
	\vdots&\vdots&\vdots&&\ddots&\vdots&0\\ 
	0&\ast&\ast&\cdots&\cdots&\ast&0\\ \hline
	-\lambda&0&0&\cdots&0&0&1\\
	\end{array}\right)=:g_1.
\end{align*}
Here we look at the $(n-2, n-2)$-submatrix of $g_1$ in the middle. Then there exists $\alpha_2 \in \OO(n-2)$ such that 
\begin{align*}
[[g]] \ni g_1
\left(\begin{array}{c|ccc|c}
1&0&\cdots&0&0\\ \hline
0&&&&0\\
\vdots&&\alpha_2&&\vdots\\
0&&&&0\\ \hline
0&0&\cdots&0&1
\end{array}\right)=
\left(\begin{array}{c|ccccc|c}
	\ast&\ast&\cdots&\cdots&\ast&a_n&0\\ \hline
	\ast&\ast&\cdots&\cdots&\ast&a_{n-1}&0\\ 
	\ast&0&\ddots&&\vdots&\vdots&\vdots\\
	0&\vdots&&\ddots&\vdots&\vdots&\vdots\\
	\vdots&\vdots&&&\ast&a_3&\vdots\\ 
	0&0&\cdots&\cdots&0&a&0\\ \hline
	-\lambda&0&0&\cdots&0&0&1\\
	\end{array}\right)=:g_2,
\end{align*}
where $a_3, \ldots, a_n \in \R$. One knows $a \neq 0$ since $\det(g_2) \neq 0$. Then it follows from the definition of $H'$ that
\begin{align}
\label{matrix : h_1}
h_1:=\left(\begin{array}{cc|ccc|c|c}
	1&0&0&\cdots&0&-a_n/a&0\\
	0&1&0&\cdots&0&-a_{n-1}/a&0\\ \hline
	0&0&1&&0&-a_{n-2}/a&0\\
	\vdots&\vdots&&\ddots&&\vdots&\vdots\\
	0&0&0&&1&-a_3/a&0\\ \hline
	0&0&0&\cdots&0&1/a&0\\ \hline
	0&0&0&\cdots&0&0&1\\
	\end{array}\right) \in H'.
\end{align}
By multiplying this matrix, one can directly see that
\begin{align*}
[[g]] \ni h_1g_2=
\left(\begin{array}{c|ccccc|c}
	\ast&\ast&\cdots&\cdots&\ast&0&0\\ \hline
	\ast&\ast&\cdots&\cdots&\ast&0&0\\ 
	\ast&0&\ddots&&\vdots&\vdots&\vdots\\
	0&\vdots&&\ddots&\vdots&\vdots&\vdots\\
	\vdots&\vdots&&&\ast&0&\vdots\\ 
	0&0&\cdots&\cdots&0&1&0\\ \hline
	-\lambda&0&0&\cdots&0&0&1\\
	\end{array}\right)=:g_3.
\end{align*}
By repeating the same procedure, one can see that there exists $h_2 \in H'$ such that
\begin{align*}
[[g]] \ni h_2g_3=
\left(\begin{array}{ccc|cccc}
	\ast&\ast&\ast&0&\cdots&\cdots&0\\
	\ast&\ast&\ast&0&\cdots&\cdots&0\\
	y&0&x&0&\cdots&\cdots&0\\ \hline
	0&0&0&1&&&0\\ 
	\vdots&\vdots&\vdots&&\ddots&&\\
	0&0&0&&&\ddots&\\
	-\lambda&0&0&0&&&1
\end{array}\right)=:g_4,
\end{align*}
where $x, y \in \R$. Therefore, in order to prove the lemma, we have only to consider the case of $n=4$, that is, 
\begin{align}
\label{matrix : x, y}
g_4':=\left(\begin{array}{ccc|c}
		\ast&\ast&\ast&0\\ 
		\ast&\ast&\ast&0\\
		y&0&x&0\\ \hline
		-\lambda&0&0&1
	\end{array}\right),
\end{align}
since the remaining blocks of $g_4$ do not have to be changed throughout the following calculations. 

We here show that one can assume $x \neq 0$ without loss of generality. In order to prove this, assume that $x=0$. Note that $y \neq 0$ holds, since $\det g_4' \neq 0$. Let us put
\begin{align*}
k_1:=\left(\begin{array}{ccc|c}
	0&0&\sqrt{\lambda^2+1}&\lambda\\
	0&1&0&0\\
	-\sqrt{\lambda^2+1}&0&\lambda^2&\lambda\sqrt{\lambda^2+1}\\ \hline
	-\lambda&0&\lambda\sqrt{\lambda^2+1}&\lambda^2+1
\end{array}\right) \in \OO(3, 1).
\end{align*}
One thus has
\begin{align*}
[[g]] \ni g_4'k_1=
\left(\begin{array}{ccc|c}
	\ast&\ast&\ast&\ast\\
	\ast&\ast&\ast&\ast\\
	0&0&y\sqrt{\lambda^2+1}&y\lambda\\ \hline
	-\lambda&0&0&1
\end{array}\right)=:g_5'.
\end{align*}
Here we can take $h_3 \in H'$ similar to (\ref{matrix : h_1}) such that
\begin{align*}
[[g]] \ni h_3g_5'= 
\left(\begin{array}{ccc|c}
	\ast&\ast&\ast&0\\
	\ast&\ast&\ast&0\\
	y\lambda^2&0&y\sqrt{\lambda^2+1}&0\\ \hline
	-\lambda&0&0&1
\end{array}\right)=:g_6'.
\end{align*} 
Note that $g_6'$ is of the same form as $g_4'$ obtained in (\ref{matrix : x, y}), and one knows $y\sqrt{\lambda^2+1} \neq 0$ since $y \neq 0$. This completes the proof of the claim, that is, in (\ref{matrix : x, y}) we can assume $x \neq 0$ without loss of generality. 

Now we consider $g_4'$ with $x \neq 0$. We can again take $h_4 \in H'$ similar to (\ref{matrix : h_1}) such that
\begin{align*}
[[g]]\ni h_4g_4'=\left(\begin{array}{ccc|c}
	b_1&b_2&0&0\\
	b_3&b_4&0&0\\ 
	\ast&0&1&0\\ \hline
	-\lambda&0&0&1
	\end{array}\right)=:g_5,
\end{align*}
where $b_1, b_2, b_3, b_4 \in \R$. Since $0 \neq \det g_5=b_1b_4-b_2b_3$, we can take
\begin{align*}
A:=\left(\begin{array}{cc}
	b_1&b_2\\
	b_3&b_4
	\end{array}\right)^{-1}, \quad 
h_5:=\left(\begin{array}{c|c}
	A&0\\ \hline
	0&I_2
	\end{array}\right) \in H'.
\end{align*}
We thus obtain the desired matrix $h_5g_5 \in [[g]]$. This completes the proof.
\end{proof}

For the latter arguments, we here modify the matrix given in Lemma~\ref{lem : procedure of lambda=1, 2}. Let us consider 
\begin{align*}
g:=\left(\begin{array}{c|cccccc}
	1&0&\cdots&\cdots&\cdots&\cdots&0\\ \hline
	0&&&&&&\\
	t&&&&&&\\
	0&&&\multicolumn{2}{c}{\multirow{2}{*}{$I_{n-1}$}}&&\\
	\vdots&&&&&&\\
	0&&&&&&\\
	-\lambda&&&&&&
	\end{array}\right) \in G_\lambda. 
\end{align*}
Then there exists $\alpha \in \OO(n-3)$ which maps $\trans (t, 0, \ldots, 0)$ to $\trans (0, \ldots, 0, |t|)$. Therefore we obtain 
\begin{align}
\begin{split}
\label{matrix |t|}
[[g]] &\ni \left(\begin{array}{cc|ccc|c}
	1&0&0&\cdots&0&0\\
	0&1&0&\cdots&0&0\\ \hline
	0&0&&&&0\\
	\vdots&\vdots&&\alpha&&\vdots\\
	0&0&&&&0\\ \hline
	0&0&0&\cdots&0&1\\
	\end{array}\right)g
\left(\begin{array}{cc|ccc|c}
	1&0&0&\cdots&0&0\\
	0&1&0&\cdots&0&0\\ \hline
	0&0&&&&0\\
	\vdots&\vdots&&\alpha^{-1}&&\vdots\\
	0&0&&&&0\\ \hline
	0&0&0&\cdots&0&1\\
	\end{array}\right)\\ 
&=\left(\begin{array}{c|ccccc}
	1&0&\cdots&\cdots&\cdots&0\\ \hline
	0&&&&&\\
	\vdots&&&&&\\
	0&&&I_{n-1}&&\\
	|t|&&&&&\\
	-\lambda&&&&&
	\end{array}\right).
\end{split}
\end{align}

We need to study the double cosets $[[g]]$ for $g \in G_\lambda$ with $\lambda \in \{1, 2\}$. We here study the case of $\lambda=1$. In this case there are two possibilities.

\begin{Prop}
\label{lem : lambda=1}
Let $g \in G_1$. Then there exists $\xi \in \{0, 1\}$ such that 
\begin{align*}
\left(
		\begin{array}{cccc}
		1&&&\\ 
		&\ddots&&\\ 
		\xi&&1&\\ 
		-1&&&1
		\end{array}
	\right) \in [[g]].
\end{align*}
\end{Prop}

\begin{proof}
Take any $g \in G_1$. Then by (\ref{matrix |t|}), there exists $t \geq 0$ such that $g$ can be turned into
\begin{align*}
g_1:=\left(\begin{array}{cc|c|cc}
	1&0&&0&0\\ 
	0&1&&0&0\\ \hline
	&&I_{n-4}&&\\ \hline
	t&0&&1&0\\
	-1&0&&0&1
	\end{array}\right) \in [[g]].
\end{align*}
If $t=0$, then it corresponds to the case of $\xi=0$. Hence we have only to consider the case of $t>0$. Furthermore, we have only to consider the case of $n=4$, that is, 
\begin{align*}
g_2:=\left(\begin{array}{cc|cc}
		1&0&0&0\\ 
		0&1&0&0\\ \hline
		t&0&1&0\\ 
		-1&0&0&1
		\end{array}
	\right) \in [[g]],
\end{align*}
since $g_2$ consists of the four $(2 \times 2)$-blocks of the four corners of $g_1$, and the remaining blocks of $g_1$ do not have to be changed.

We show that $[[g_2]]$ contains the matrix in the claim with $\xi=1$. Let us put  
\begin{align*}
s:=(t-1)/t, \quad
k_1:=\left(\begin{array}{cc|cc}
	1-s^2/2&0&s&s^2/2\\ 
	0&1&0&0\\ \hline
	-s&0&1&s\\ 
	-s^2/2&0&s&1+s^2/2
	\end{array}\right) \in \OO(3, 1).
\end{align*}
Then a direct calculation yields that
\begin{align*}
[[g]] \ni g_2k_1=\left(\begin{array}{cc|cc}
	1-s^2/2&0&s&s^2/2\\ 
	0&1&0&0\\ \hline
	t-s^2t/2-s&0&1+st&s^2t/2+s\\ 
	-1&0&0&1
	\end{array}\right)=:g_3.
\end{align*}
We next take 
\begin{align*}
h_1:=\left(\begin{array}{cc|cc}
	1&0&0&-s^2/2\\ 
	0&1&0&0\\ \hline
	0&0&1&-s^2t/2-s\\ 
	0&0&0&1
	\end{array}\right) \in H'.
\end{align*}
Then one obtains
\begin{align*}
[[g]] \ni h_1g_3
=\left(\begin{array}{cc|cc}
	1&0&s&0\\ 
	0&1&0&0\\ \hline
	t&0&1+st&0\\ 
	-1&0&0&1
	\end{array}\right)=:g_4.
\end{align*}
We here note that
\begin{align*}
1+st=1+t-1=t \neq 0.
\end{align*}
Hence one can take
\begin{align*}
h_2:=\left(\begin{array}{cc|cc}
	t&0&-s&0\\ 
	0&1&0&0\\ \hline
	0&0&1/t&0\\ 
	0&0&0&1
	\end{array}\right) \in H'.
\end{align*}
We then consider $h_2g_4 \in [[g]]$, and a direct calculation yields that this corresponds to the desired matrix with $\xi=1$.
\end{proof}

The last case is $\lambda=2$, that is, we study the double cosets $[[g]]$ for $g \in G_2$. In this case there are three possibilities.

\begin{Prop}
\label{lem : lambda=2}
Let $g \in G_2$. Then there exists $\xi \in \{0, \sqrt{3}, 2\}$ such that 
\begin{align*}
\left(
		\begin{array}{cccc}
		1&&&\\ 
		&\ddots&&\\ 
		\xi&&1&\\ 
		-2&&&1
		\end{array}
	\right) \in [[g]].
\end{align*}
\end{Prop}

\begin{proof}
Take any $g \in G_2$. Then, by (\ref{matrix |t|}), there exists $t \geq 0$ such that
\begin{align*}
g_1:=\left(\begin{array}{cc|c|cc}
	1&0&&0&0\\ 
	0&1&&0&0\\ \hline
	&&I_{n-4}&&\\ \hline
	t&0&&1&0\\
	-2&0&&0&1
	\end{array}\right) \in [[g]].
\end{align*}
Similar to the proof of Proposition~\ref{lem : lambda=1}, we have only to consider the case of $n=4$, that is, 
\begin{align*}
g_2:=\left(\begin{array}{cc|cc}
		1&0&0&0\\ 
		0&1&0&0\\ \hline
		t&0&1&0\\ 
		-2&0&0&1
		\end{array}
	\right) \in [[g]].
\end{align*}

If $t=\sqrt{3}$, then it corresponds to the case of $\xi=\sqrt{3}$. We need to study the case of $t \neq \sqrt{3}$. For this purpose, let us put
\begin{align*}
\varphi(s):=\sqrt{3s^2-8s+5} \quad (s \geq 5/3).
\end{align*}
Then one can directly check that
\begin{align*}
k_1:=\left(
		\begin{array}{cc|cc}
		s&0&-\varphi(s)&-2s+2\\ 
		0&1&0&0\\ \hline
		-\varphi(s)&0&3s-4&2\varphi(s)\\ 
		2s-2&0&-2\varphi(s)&-4s+5
		\end{array}
	\right) \in \OO(3, 1).
\end{align*}
By multiplying $k_1$ from the right, we have
\begin{align*}
[[g]] \ni g_2k_1=\left(
		\begin{array}{cc|cc}
		s&0&-\varphi(s)&-2s+2\\ 
		0&1&0&0\\ \hline
		st-\varphi(s)&0&-t\varphi(s)+3s-4&-2st+2t+2\varphi(s)\\ 
		-2&0&0&1
		\end{array}
	\right)=:g_3.
\end{align*}
Similar to the previous arguments, we consider the following element in $H'$ :
\begin{align*}
h_1:=\left(\begin{array}{cc|cc}
	1&0&0&2s-2\\
	0&1&0&0\\ \hline
	0&0&1&2st-2t-2\varphi(s)\\ 
	0&0&0&1
	\end{array}\right) \in H'.
\end{align*}
Then one has
\begin{align*}
[[g]] \ni h_1g_3=
\left(
		\begin{array}{cc|cc}
		-3s+4&0&-\varphi(s)&0\\ 
		0&1&0&0\\ \hline
		3\varphi(s)-t(3s-4)&0&-t\varphi(s)+3s-4&0\\ 
		-2&0&0&1
		\end{array}
	\right)=:g_4.
\end{align*}

Remember $t \geq 0$, and we have to study $t \neq \sqrt{3}$. We divide the following argument into two cases. The first case is $0 \leq t < \sqrt{3}$. In this case, the $(3, 1)$-component of $g_4$ can be zero, that is, the following equation on $s$ has a solution:
\begin{align}
\label{eq : intermediate value thm 1}
3\varphi(s)-t(3s-4)=0.
\end{align}
In fact, the left hand side has a negative value when $s=5/3$ and has a positive value as $s$ approaches to $+\infty$. Hence by the intermediate value theorem, this equation has a solution $s=s_0$. Therefore, by substituting the solution $s=s_0$ into $g_4$, its $(3, 1)$-component is zero. By multiplying suitable element of $H'$ from the left, one can show that $[[g]]$ contains the desired matrix with $\xi=0$.

It remains to study the case of $t>\sqrt{3}$. Similarly to the previous case, the intermediate value theorem yields that the following has a solution:
\begin{align}
\label{eq : intermediate value thm 2}
3\varphi(s)-t(3s-4)=2(-t\varphi(s)+3s-4).
\end{align}
Therefore, by substituting the solution $s=s_1$ into $g_4$, its $(3, 1)$-component is $2(-t\varphi(s)+3s-4)$, which is equal to double of the $(3, 3)$-component. By multiplying suitable element of $H'$ from the left, one can show that $[[g]]$ contains the desired matrix with $\xi=2$. This completes the proof.
\end{proof}

The next proposition gives a set of representatives of the action of $\R^\times \Aut(\LG)$ on $\LLM_{(n-1, 1)}(\LG)$.

\begin{Prop}
\label{prop : set of rep}
We put $\llu:=\{(0, 0), (1, 0), (1, 1), (2, 0), (2, \sqrt{3}), (2, 2)\}$. The following $\LLU$ is a set of representatives of the action of $\R^\times \Aut(\LG)$ on $\LLM_{(n-1, 1)}(\LG)$ $:$ 
\begin{align*}
\LLU:=\left\{
	\left(
		\begin{array}{cccc}
		1&&\xi&\lambda\\ 
		&\ddots&&\\ 
		&&1&\\ 
		&&&1
		\end{array}
	\right) \middle| (\lambda, \xi) \in \llu \right\}.
\end{align*}
\end{Prop}

\begin{proof}
From Lemma~\ref{lem : lambda=0, 1, 2} and Propositions~\ref{lem : lambda=0}, \ref{lem : lambda=1} and \ref{lem : lambda=2}, it immediately follows that the following $\LLU^\ast$ is a set of representatives of the action of $H'$ on $\LLM_{(n-1, 1)}(\LG):$
\begin{align*}
\LLU^\ast:=\left\{
	\left(
		\begin{array}{cccc}
		1&&&\\ 
		&\ddots&&\\ 
		\xi&&1&\\ 
		-\lambda&&&1
		\end{array}
	\right) \middle| (\lambda, \xi) \in \llu \right\}.
\end{align*}
One thus has from Lemma~\ref{lem : action of transpose} and the definition of $H'$ that
\begin{align*}
\LLU':=\left\{
	\left(
		\begin{array}{cccc}
		1&&-\xi&\lambda\\ 
		&\ddots&&\\ 
		&&1&\\ 
		&&&1
		\end{array}
	\right) \middle| (\lambda, \xi) \in \llu \right\}
\end{align*}
is a set of representatives of the action of $\RAutg$ on $\LLM_{(n-1, 1)}(\LG)$. Moreover, we consider
\begin{align*}
\left(\begin{array}{ccccc}
	1&&&&\\
	&\ddots&&&\\
	&&1&&\\
	&&&-1&\\
	&&&&1
	\end{array}\right) \in \RAutg, \ \OO(n-1, 1).
\end{align*}
By multiplying this matrix to each element in $\LLU'$ from the both sides, one can change the part $(-\xi, \lambda)$ into $(\xi, \lambda)$, which completes the proof.
\end{proof}

\section{A Milnor-type theorem and a classification of inner products}
\label{sec4}

In this section, we obtain a Milnor-type theorem for inner products of signature $(n-1, 1)$ on $\LG:=\LH_3 \oplus \R^{n-3}$ with $n \geq 4$, and show that there exist exactly six such inner products up to scaling and automorphisms. Recall that the set of representatives is parametrized by 
\begin{align*}
\llu:=\{(0, 0), (1, 0), (1, 1), (2, 0), (2, \sqrt{3}), (2, 2)\}.
\end{align*}

\subsection{A Milnor-type theorem}

In this subsection, we obtain a Milnor-type theorem for $\LG:=\LH_3 \oplus \R^{n-3}$ with $n \geq 4$. First of all, we give some change of basis for the latter use. Let $\{e_1, \ldots, e_n\}$ be the standard basis of $\LG:=\LH_3 \oplus \R^{n-3}$, and $\hoge< , >_0$ be the canonical inner product on $\LG$ with signature $(n-1, 1)$.

\begin{Lem}
\label{lem : bracket of x'}
Let $\lambda, \xi \in \R$, and define
\begin{align*}
g_{\lambda, \xi}:=
	\left(
		\begin{array}{cccc}
		1&&\xi&\lambda\\ 
		&\ddots&&\\ 
		&&1&\\ 
		&&&1
		\end{array}
	\right), \hspace{5mm} \hoge< , >_{\lambda, \xi}:=g_{\lambda, \xi}.\hoge< , >_0, \hspace{5mm} x_i':=g_{\lambda, \xi}e_i.
\end{align*}
Then $\{x_1', \ldots, x_n'\}$ is a pseudo-orthonormal basis of $\LG$ with respect to $\hoge< , >_{\lambda, \xi}$, and the bracket relation among them is given by
\begin{align*}
[x_1', x_2']=-(\lambda x_1'-x_n'), \hspace{5mm} [x_2', x_{n-1}']=\xi(\lambda x_1'-x_n'), \hspace{5mm} [x_2', x_n']=\lambda(\lambda x_1'-x_n').
\end{align*}
\end{Lem}

\begin{proof}
The first assertion is obvious since $\{e_1, \ldots, e_n\}$ is pseudo-orthonormal with respect to $\hoge< , >_0$. We prove the second assertion on the bracket relation. By the definition of $g_{\lambda, \xi}$, it is easy to see that
\begin{align*}
%\label{eq : bracket1}
x_i'&=g_{\lambda, \xi}e_i=e_i \hspace{5mm} (i \in \{1, \ldots, n-2\}),\\
%\label{eq : bracket2}
x_{n-1}'&=g_{\lambda, \xi}e_{n-1}=\xi e_1+e_{n-1},\\
%\label{eq : bracket3}
x_n'&=g_{\lambda, \xi}e_n=\lambda e_1+e_n.
\end{align*}
Note that one has 
\begin{align*}
e_n=-\lambda e_1+(\lambda e_1+e_n)=-\lambda x_1'+x_n'.
\end{align*}
Therefore, by the bracket relation $[e_1, e_2]=e_n$ with respect to the standard basis, we can see that
\begin{align*}
[x_1', x_2']&=[e_1, e_2]=e_n=-(\lambda x_1'-x_n'),\\
[x_2', x_{n-1}']&=[e_2, \xi e_1+e_{n-1}]=-\xi e_n=\xi(\lambda x_1'-x_n'),\\
[x_2', x_n']&=[e_2, \lambda e_1+e_n]=-\lambda e_n=\lambda(\lambda x_1'-x_n').
\end{align*}
One can also see that the other bracket relations precisely vanish. This completes the proof of this lemma.
\end{proof}

In terms of this lemma and a description of the set of representatives given in Proposition~\ref{prop : set of rep}, we can obtain the following Milnor-type theorem.

\begin{Thm}
\label{thm : milnor}
Let $\hoge< , >$ be an inner product of signature $(n-1, 1)$ on $\LH_3 \oplus \R^{n-3}$ with $n \geq 4$. Then, there exist $k>0$, $(\lambda, \xi) \in \llu$ and a pseudo-orthonormal basis $\{x_1, \ldots, x_n\}$ with respect to $k\hoge< , >$ such that the bracket relation is given by
\begin{align*}
[x_1, x_2]=-(\lambda x_1-x_n), \quad [x_2, x_{n-1}]=\xi(\lambda x_1-x_n), \quad [x_2, x_n]=\lambda(\lambda x_1-x_n).
\end{align*}
\end{Thm}

\begin{proof}
Take an inner product $\hoge< , >$ of signature $(n-1, 1)$ on $\LG:=\LH_3 \oplus \R^{n-3}$. By Proposition~\ref{prop : set of rep} we know that
\begin{align*}
\LLU:=\left\{
	g_{\lambda, \xi}=\left(
		\begin{array}{cccc}
		1&&\xi&\lambda\\ 
		&\ddots&&\\ 
		&&1&\\ 
		&&&1
		\end{array}
	\right) \middle| (\lambda, \xi) \in \llu\right\}
\end{align*}
is a set of representatives of the action of $\R^\times \Aut(\LG)$ on $\LLM_{(n-1, 1)}(\LG)$ with respect to the standard basis $\{e_1, \ldots, e_n\}$. Hence, we have from Theorem~\ref{thm : procedure} that there exist $k>0$, $\varphi \in \Aut(\LG)$ and $(\lambda, \xi) \in \llu$ such that $\{\varphi g_{\lambda, \xi}e_1, \ldots, \varphi g_{\lambda, \xi}e_n\}$ is pseudo-orthonormal with respect to $k\hoge< , >$. Let us put 
\begin{align*}
x_i':=g_{\lambda, \xi}e_i, \quad x_i:=\varphi x_i' \quad (i \in \{1, \ldots, n\}). 
\end{align*}
The bracket relation among $\{x_1', \ldots, x_n'\}$ is given in Lemma~\ref{lem : bracket of x'}. 
Then we obtain the bracket relation among $\{x_1, \ldots, x_n\}$, which is of the same form, since $\varphi$ is an automorphism. 
\end{proof}

\subsection{A classification of inner products}

In order to give a classification of inner products of signature $(n-1, 1)$ on $\LG$ up to scaling and automorphisms, we have to distinguish the above six inner products. Fix an inner product $\hoge< , >$ of signature $(n-1, 1)$ of $\LG$ and let $\{x_1, \ldots, x_n\}$ be a basis given in Theorem~\ref{thm : milnor}. By Theorem~\ref{thm : milnor}, the center $Z(\LG)$ and the derived ideal $[\LG, \LG]$ of $\LG$ can be expressed as 
\begin{align*}
Z(\LG)&=\s\{x_3, \ldots, x_{n-2}, \xi x_1-x_{n-1}, \lambda x_1-x_n\},\\
[\LG, \LG]&=\s\{\lambda x_1-x_n\}.
\end{align*}

\begin{Lem}
\label{lem : basis of center}
We put
\begin{align*}
y_i&:=x_{i+2} \quad (i \in \{1, \ldots, n-4\}),\\
y_{n-3}&:=(1/\sqrt{\xi^2+1})(\xi x_1-x_{n-1}),\\
y_{n-2}&:=\lambda(x_1+\xi x_{n-1})-(\xi^2+1)x_n.
\end{align*}
Then, $\{y_1, \ldots, y_{n-2}\}$ is an orthogonal basis of $Z(\LG)$ with respect to $\hoge< , >$. Especially, it satisfies $\hoge<y_i, y_i>=1$ for each $i=1, \ldots, n-3$.
\end{Lem}

\begin{proof}
It follows from direct and easy calculations.
\end{proof}

We then consider the signatures of the restrictions of $\hoge< , >$ to the center $Z(\LG)$ and the derived ideal $[\LG, \LG]$. In fact, these data characterize inner products up to scaling and automorphisms. We denote by $[\hoge< , >_{\lambda, \xi}]$ the equivalence class of $\hoge< , >_{\lambda, \xi}$ up to scaling and automorphisms for each $(\lambda, \xi) \in \llu$.

\begin{Prop}
\label{prop : signature on subalgebra}
Let $\hoge< , >$ be an inner product of signature $(n-1, 1)$ on $\LG$. If $\hoge< , > \in [\hoge< , >_{\lambda, \xi}]$, then its restrictions on $Z(\LG)$ and $[\LG, \LG]$ have the signatures given in Table~\ref{table : signature}, where the signature convention is $(+, -, 0):$
\begin{table}[h]
	\begin{tabular}{|c|c|c|}\hline
	$(\lambda, \xi)$&signature\ on $Z(\LG)$&signature\ on $[\LG, \LG]$\\\hline \hline
	$(0, 0)$&$(n-3, 1, 0)$&$(0, 1, 0)$\\ \hline
	$(1, 0)$&$(n-3, 0, 1)$&$(0, 0, 1)$\\ \hline
	$(1, 1)$&$(n-3, 1, 0)$&$(0, 0, 1)$\\ \hline
	$(2, 0)$&$(n-2, 0, 0)$&$(1, 0, 0)$\\ \hline
	$(2, \sqrt{3})$&$(n-3, 0, 1)$&$(1, 0, 0)$\\ \hline
	$(2, 2)$&$(n-3, 1, 0)$&$(1, 0, 0)$\\ \hline
	\end{tabular}
\caption{Signatures on the subalgebras}
\label{table : signature}
\end{table}
\end{Prop}

\begin{proof}
First of all, we consider the signature of $\hoge< , >':=\hoge< , >\mid_{Z(\LG) \times Z(\LG)}$. We use the basis $\{y_1, \ldots, y_{n-2}\}$ of $Z(\LG)$ given in Lemma~\ref{lem : basis of center}. Since it is orthogonal, the representation matrix of $\hoge< , >'$ is diagonal. One knows $\hoge<y_i, y_i>'=1$ for each $i \in \{1, \ldots, n-3\}$. Hence, in order to determine the signature, we have only to compute $\hoge<y_{n-2}, y_{n-2}>'$. One has
\begin{align*}
\hoge<y_{n-2}, y_{n-2}>'&=\hoge<\lambda(x_1+\xi x_{n-1})-(\xi^2+1)x_n, \lambda(x_1+\xi x_{n-1})-(\xi^2+1)x_n>\\
&=\lambda^2\hoge<x_1, x_1>+\lambda^2\xi^2\hoge<x_{n-1}, x_{n-1}>+(\xi^2+1)^2\hoge<x_n, x_n>\\
&=\lambda^2+\lambda^2\xi^2-(\xi^2+1)^2\\
&=(\xi^2+1)(\lambda^2-\xi^2-1).
\end{align*}
Then one can easily see that 
\begin{align*}
&\hoge<y_{n-2}, y_{n-2}>'<0 \quad (\mathrm{if}\ (\lambda, \xi)=(0, 0), (1, 1), (2, 2)),\\
&\hoge<y_{n-2}, y_{n-2}>'=0 \quad (\mathrm{if}\ (\lambda, \xi)=(1, 0), (2, \sqrt{3})),\\
&\hoge<y_{n-2}, y_{n-2}>'>0 \quad (\mathrm{if}\ (\lambda, \xi)=(2, 0)).
\end{align*}
According to these three cases, the signature of $\hoge< , >'$ is $(n-3, 1, 0)$, $(n-3, 0, 1)$,
and $(n-2, 0, 0)$, respectively. This proves the left column of Table~\ref{table : signature}.

It remains to study the signature of $\hoge< , >\mid_{[\LG, \LG] \times [\LG, \LG]}$. Remember that $[\LG, \LG]=\s\{\lambda x_1-x_n\}$. One has
\begin{align*}
\hoge<\lambda x_1-x_n, \lambda x_1-x_n>=\lambda^2\hoge<x_1, x_1>+\hoge<x_n, x_n>=\lambda^2-1.
\end{align*}
Hence, if $\lambda=0, 1, 2$, then $\hoge< , >$ is negative definite, degenerate, positive definite on $[\LG, \LG]$, respectively. These complete the proof.
\end{proof}

Finally in this section, we prove Theorem~\ref{thm : classification of inner products}, which classifies left-invariant Lorentzian metrics on $H_3 \times \R^{n-3}$ up to scaling and automorphisms.

\begin{proof}[Proof of Theorem~\ref{thm : classification of inner products}]
We show that the action of $\RAutg$ has exactly six orbits. It follows from Proposition~\ref{prop : set of rep} that the number of orbits is at most six. Thus we have only to show that each $\hoge< , >_{\lambda, \xi}$ is in a distinct orbit. Note that the action of $\RAutg$ preserves $Z(\LG)$ and $[\LG, \LG]$. Also, it preserves the signatures of the restrictions of each inner product on these two subspaces by Sylvester's law of inertia. Hence, we obtain from Proposition~\ref{prop : signature on subalgebra} that there exist exactly six inner products of signature $(n-1, 1)$ on $\LH_3 \oplus \R^{n-3}$ with $n \geq 4$ up to scaling and automorphisms, which completes the proof of  Theorem~\ref{thm : classification of inner products}.
\end{proof}

\section{Ricci soliton and flat metrics} 
\label{sec5} 

In this section, we calculate curvatures of an arbitrary left-invariant Lorentzian metric $\hoge< , >$ on $\LG:=\LH_3 \oplus \R^{n-3}$, and prove Theorem~\ref{thm : Ricci and Einstein}, which is the second main result. It follows from Theorem~\ref{thm : milnor} that there exist $k>0$,  
\begin{align*}
(\lambda, \xi) \in \llu:=\{(0, 0), (1, 0), (1, 1), (2, 0), (2, \sqrt{3}), (2, 2)\},
\end{align*}
and a pseudo-orthonormal basis $\{x_1, \ldots, x_n\}$ with respect to $k\hoge< , >$ whose bracket relations are given by
\begin{align*}
[x_1, x_2]=-(\lambda x_1-x_n), \quad [x_2, x_{n-1}]=\xi(\lambda x_1-x_n), \quad [x_2, x_n]=\lambda(\lambda x_1-x_n).
\end{align*}
Throughout the following arguments, we calculate the curvatures in terms of this basis, under the normalization $k=1$ for simplicity.

\subsection{Calculations of curvatures}

In this subsection, we calculate the curvatures of $(\LG, \hoge< , >)$. We here note that $\s\{x_3, \ldots, x_{n-2}\}$ is a subspace of $Z(\LG)$, which is orthogonal to the derived subalgebra $[\LG, \LG]$. Thus $\{x_3, \ldots, x_{n-2}\}$ does not give any influence on calculation of curvatures. Hence we calculate curvatures only for $\{x_1, x_2, x_{n-1}, x_n\}$.

First of all, we calculate the symmetric bilinear map $U: \LG \times \LG \to \LG$ of $(\LG, \hoge< , >)$ given by
%\vspace{-5mm}
\begin{align*}
2\hoge<U(X, Y), Z>=\hoge<[Z, X], Y>+\hoge<X, [Z, Y]> \quad (\forall X, Y, Z \in \LG).
\end{align*}

\begin{Lem}
\label{lem : U}
The map $U: \LG \times \LG \to \LG$ of $(\LG, \hoge< , >)$ satisfies the following$:$
\begin{itemize}
\setlength{\itemsep}{5pt}
\item[(1)]$U(x_1, x_1)=\lambda x_2$,\hspace{5mm} $U(x_1, x_2)=-(\lambda/2)x_1-(\lambda\xi/2)x_{n-1}+(\lambda^2/2)x_n$,\\ $U(x_1, x_{n-1})=(\lambda\xi/2)x_2$,\hspace{5mm} $U(x_1, x_n)=((\lambda^2+1)/2)x_2$,\ 
\item[(2)]$U(x_2, x_2)=0$,\hspace{5mm} $U(x_2, x_{n-1})=0$,\\ $U(x_2, x_n)=-(1/2)x_1-(\xi/2)x_{n-1}+(\lambda/2)x_n$,\ 
\item[(3)]$U(x_{n-1}, x_{n-1})=0$,\hspace{5mm} $U(x_{n-1}, x_n)=(\xi/2)x_2$,
\item[(4)]$U(x_n, x_n)=\lambda x_2$.
\end{itemize}
\end{Lem}

\begin{proof}
We show the calculation only for the case of $U(x_1, x_2)$. By the definition of $U$, we obtain
\begin{align*}
&2\hoge<U(x_1, x_2), x_1>=\hoge<x_1, -(\lambda x_1-x_n)>=-\lambda,\\
&2\hoge<U(x_1, x_2), x_2>=\hoge<\lambda x_1-x_n, x_2>=0,\\
&2\hoge<U(x_1, x_2), x_{n-1}>=\hoge<x_1, -\xi(\lambda x_1-x_n)>=-\lambda\xi,\\
&2\hoge<U(x_1, x_2), x_n>=\hoge<x_1, -\lambda(\lambda x_1-x_n)>=-\lambda^2.
\end{align*}
Hence one has the expression of $U(x_1, x_2)$ in terms of the basis $\{x_1, \ldots, x_n\}$. We can similarly prove the remaining.
\end{proof}

Next, we calculate the Levi-Civita connection $\nabla: \LG \times \LG \to \LG$ of $(\LG, \hoge< , >)$ defined by
\begin{align*}
\nabla_XY:=(1/2)[X, Y]+U(X, Y) \quad (\forall X, Y \in \LG).
\end{align*}
Note that $\nabla$ is bilinear, but neither symmetric nor skew-symmetric.

\begin{Lem}
\label{lem : nabla}
The Levi-Civita connection $\nabla: \LG \times \LG \to \LG$ of $(\LG, \hoge< , >)$ satisfies the following$:$
\begin{itemize}
\setlength{\itemsep}{5pt}
\item[(1)]$\nabla_{x_1}x_1=\lambda x_2$,\hspace{5mm} $\nabla_{x_1}x_2=-\lambda x_1-(\lambda\xi/2)x_{n-1}+((\lambda^2+1)/2)x_n$,\\
$\nabla_{x_1}x_{n-1}=(\lambda\xi/2)x_2$,\hspace{5mm} $\nabla_{x_1}x_n=((\lambda^2+1)/2)x_2$,
\item[(2)]$\nabla_{x_2}x_1=-(\lambda\xi/2)x_{n-1}+((\lambda^2-1)/2)x_n$,\hspace{5mm} $\nabla_{x_2}x_2=0$,\\ 
$\nabla_{x_2}x_{n-1}=(\xi/2)(\lambda x_1-x_n)$,\hspace{5mm} $\nabla_{x_2}x_n=((\lambda^2-1)/2)x_1-(\xi/2)x_{n-1}$,
\item[(3)]$\nabla_{x_{n-1}}x_1=(\lambda\xi/2)x_2$,\hspace{5mm} $\nabla_{x_{n-1}}x_2=-(\xi/2)(\lambda x_1-x_n)$,\\ 
$\nabla_{x_{n-1}}x_{n-1}=0$,\hspace{5mm} $\nabla_{x_{n-1}}x_n=(\xi/2)x_2$,
\item[(4)]$\nabla_{x_n}x_1=((\lambda^2+1)/2)x_2$,\hspace{5mm} $\nabla_{x_n}x_2=-((\lambda^2+1)/2)x_1-(\xi/2)x_{n-1}+\lambda x_n$,\\ 
$\nabla_{x_n}x_{n-1}=(\xi/2)x_2$,\hspace{5mm} $\nabla_{x_n}x_n=\lambda x_2$.
\end{itemize}
\end{Lem}

\begin{proof}
We show the calculation only for a part. One knows
\begin{align*}
(1/2)[x_1, x_2]&=-(1/2)(\lambda x_1-x_n),\\
U(x_1, x_2)&=-(\lambda/2)x_1-(\lambda\xi/2)x_{n-1}+(\lambda^2/2)x_n.
\end{align*}
Then we can calculate $\nabla_{x_1}x_2$ by taking the sum of them, and $\nabla_{x_2}x_1$ by taking the difference. We can similarly prove the remaining.
\end{proof}

Next, we calculate the curvature tensor $R: \LG \times \LG \to \gl(\LG)$ of $(\LG, \hoge< , >)$, taken with the sign convention
%\vspace{-5mm}
\begin{align*}
R(X, Y)=[\nabla_X, \nabla_Y]-\nabla_{[X, Y]} \quad (\forall X, Y \in \LG),
\end{align*}
where $\gl(\LG):=\{f: \LG \to \LG \mid f \mathrm{\ is\ linear}\}$.

\begin{Lem}
\label{lem : curvature tensor}
The curvature tensor $R: \LG \times \LG \to \gl(\LG)$ of $(\LG, \hoge< , >)$ satisfies the following$:$
\setlength{\leftmargini}{20pt}
\begin{itemize}
\setlength{\itemsep}{5pt}
\item[(1)]$4R(x_1, x_2)x_1=(\lambda^4-\lambda^2(\xi^2-2)-3)x_2$,\\ 
$4R(x_1, x_2)x_2=-(\lambda^4-\lambda^2(\xi^2-2)-3)x_1-3\xi(\lambda^2-1)x_{n-1}+(4\lambda^3-\lambda(\xi^2+4))x_n$,\\
$4R(x_1, x_2)x_{n-1}=3\xi(\lambda^2-1)x_2$,\hspace{5mm} $4R(x_1, x_2)x_n=(4\lambda^3-\lambda(\xi^2+4))x_2$,
\item[(2)]$4R(x_1, x_{n-1})x_1=-\lambda^2\xi^2x_{n-1}+\lambda\xi(\lambda^2-1)x_n$,\hspace{5mm} $R(x_1, x_{n-1})x_2=0$,\\
$4R(x_1, x_{n-1})x_{n-1}=\lambda\xi^2(\lambda x_1-x_n)$,\\
$4R(x_1, x_{n-1})x_n=\lambda\xi(\lambda^2-1)x_1-\lambda\xi^2x_{n-1}$,
\item[(3)]$4R(x_1, x_n)x_1=-\lambda\xi(\lambda^2-1)x_{n-1}+(\lambda^2-1)^2x_n$,\hspace{5mm} $R(x_1, x_n)x_2=0$,\\
$4R(x_1, x_n)x_{n-1}=\xi(\lambda^2-1)(\lambda x_1-x_n)$,\\ 
$4R(x_1, x_n)x_n=(\lambda^2-1)^2x_1-\xi(\lambda^2-1)x_{n-1}$,
\item[(4)]$4R(x_2, x_{n-1})x_1=-3\xi(\lambda^2-1)x_2$,\\
$4R(x_2, x_{n-1})x_2=3\xi(\lambda^2-1)(x_1+\xi x_{n-1}-\lambda x_n)$,\\
$4R(x_2, x_{n-1})x_{n-1}=-3\xi^2(\lambda^2-1)x_2$,\hspace{5mm} $4R(x_2, x_{n-1})x_n=-3\lambda\xi(\lambda^2-1)x_2$,
\item[(5)]$4R(x_2, x_n)x_1=-(4\lambda^3-\lambda(\xi^2+4))x_2$,\\ 
$4R(x_2, x_n)x_2=(4\lambda^3-\lambda(\xi^2+4))x_1+3\lambda\xi(\lambda^2-1)x_{n-1}-(3\lambda^4-2\lambda^2-\xi^2-1)x_n$,\\
$4R(x_2, x_n)x_{n-1}=-3\lambda\xi(\lambda^2-1)x_2$,\hspace{5mm} $4R(x_2, x_n)x_n=-(3\lambda^4-2\lambda^2-\xi^2-1)x_2$,
\item[(6)]$4R(x_{n-1}, x_n)x_1=\lambda\xi^2x_{n-1}-\xi(\lambda^2-1)x_n$,\hspace{5mm} $R(x_{n-1}, x_n)x_2=0$,\\
$4R(x_{n-1}, x_n)x_{n-1}=-\xi^2(\lambda x_1-x_n)$,\\ 
$4R(x_{n-1}, x_n)x_n=-\xi(\lambda^2-1)x_1+\xi^2x_{n-1}$.
\end{itemize}
\end{Lem}

\begin{proof}
We show the calculation only for the case of $R(x_1, x_2)x_2$. It follows from Lemma~\ref{lem : nabla} that
\begin{align*}
4\nabla_{x_1}\nabla_{x_2}x_2&=4\nabla_{x_1}0=0, \\
-4\nabla_{x_2}\nabla_{x_1}x_2&=2\nabla_{x_2}(2\lambda x_1+\lambda\xi x_{n-1}-(\lambda^2+1)x_n)\\
&=(\lambda^2\xi^2-\lambda^4+1)x_1+(-\lambda^2\xi+\xi)x_{n-1}+(2\lambda(\lambda^2-1)-\lambda\xi^2)x_n,\\
-4\nabla_{[x_1, x_2]}x_2&=-4\nabla_{-\lambda x_1+x_n}x_2\\
&=(-2\lambda^2+2)x_1+(-2\lambda^2\xi+2\xi)x_{n-1}+(2\lambda^3-2\lambda)x_n.
\end{align*}
One can calculate $R(x_1, x_2)x_2$ by summing up them. We can similarly prove the remaining.
\end{proof}

Finally, we calculate the Ricci curvature $\Ric: \LG \to \LG$ of $(\LG, \hoge< , >)$ given by
\begin{align*}
\textstyle \Ric(X):=\sum_{i=1}^{n-1} R(X, x_i)x_i-R(X, x_n)x_n \quad (\forall X \in \LG).
\end{align*}
Recall that $\{x_1, \ldots, x_n\}$ is a pseudo-orthonormal basis of $\LG$ with respect to $\hoge< , >$.

\begin{Lem}
\label{lem : ric}
The Ricci curvature $\Ric: \LG \to \LG$ of $(\LG, \hoge< , >)$ satisfies the following$:$
\begin{itemize}
\setlength{\itemsep}{5pt}
\item[(1)]$2\Ric(x_1)=-(\lambda^4-\lambda^2\xi^2-1)x_1-\xi(\lambda^2-1)x_{n-1}+(2\lambda^3-\lambda(\xi^2+2))x_n$,
\item[(2)]$2\Ric(x_2)=(\lambda^4-\lambda^2(\xi^2+2)+\xi^2+1)x_2$,
\item[(3)]$2\Ric(x_{n-1})=-\xi(\lambda^2-1)x_1-\xi^2(\lambda^2-1)x_{n-1}+\lambda\xi(\lambda^2-1)x_n$,
\item[(4)]$2\Ric(x_n)=-(2\lambda^3-\lambda(\xi^2+2))x_1-\lambda\xi(\lambda^2-1)x_{n-1}+(\lambda^4-\xi^2-1)x_n$.
\end{itemize}
\end{Lem}

\begin{proof}
We show the calculation only for the case of $\Ric(x_1)$. By the definition of $\Ric$, we have
\begin{align*}
\Ric(x_1)=R(x_1, x_2)x_2+R(x_1, x_{n-1})x_{n-1}-R(x_1, x_n)x_n.
\end{align*}
By substituting the result of Lemma~\ref{lem : curvature tensor}, one can obtain the expression of $\Ric(x_1)$ in terms of the basis $\{x_1, \ldots, x_n\}$. We can similarly prove the remaining.
\end{proof}

\subsection{Curvature properties}

In this subsection, we study curvature properties of $(\LG, \hoge< , >)$, such as flat, Einstein, and Ricci soliton. First of all we recall some fundamental notions.

\begin{Def}
Let $\LG$ be a Lie algebra. Then, the following set is called the {\it derivation algebra} of $\LG$:
\begin{align*}
\Der(\LG):=\{D: \LG \to \LG: \mbox{linear} \mid \forall X, Y \in \LG, D([X, Y])=[D(X), Y]+[X, D(Y)]\}.
\end{align*}
\end{Def}

\begin{Def}
\label{def : distinguished inner products}
Let $\LG$ be a Lie algebra and $\hoge< , >$ be an inner product on it.
\begin{itemize}
\item[(i)]$(\LG, \hoge< , >)$ is called an {\it algebraic Ricci soliton} if
there exist $c \in \R$ and $D \in \Der(\LG)$ such that $\Ric=c \cdot \id+D$.
\item[(ii)]$(\LG, \hoge< , >)$ is called {\it Einstein} if there exists $c \in \R$ such that $\Ric=c \cdot \id$.
\item[(iii)]$(\LG, \hoge< , >)$ is called {\it flat} if the curvature tensor $R$ satisfies $R \equiv 0$.
\end{itemize}
\end{Def}

Note that an algebraic Ricci soliton gives rise to a Ricci soliton metric, in the following sense.

\begin{Rem}
\label{rem : relation between RS and ARS}
Let $G$ be a simply-connected Lie group whose Lie algebra is $\LG$, and $g$ be a left-invariant pseudo-Riemannian metric on $G$ corresponding to an inner product $\hoge< , >$ on $\LG$. If $(\LG, \hoge< , >)$ is an algebraic Ricci soliton, then $(G, g)$ is a Ricci soliton. This is well-known for the Riemannian case $($see \cite{Lauret2}$)$, but it also holds in the pseudo-Riemannian setting $($\cite{Onda}$)$. We refer to \cite{Cao, Topping} and references therein for Ricci soliton metrics, and also to \cite{Lauret3} for soliton geometric structures on homogeneous spaces including pseudo-Riemannian metrics.
\end{Rem}

Next we describe the matrix expression of $\Der(\LG)$ for $\LG:=\LH_3 \oplus \R^{n-3}$, with respect to the basis $\{x_1, \ldots, x_n\}$ fixed at the beginning of this section. We also use the matrix $g_{\lambda, \xi}$ given in Lemma~\ref{lem : bracket of x'}. 

\begin{Lem}%[\cite{HT}]
\label{thm : matrix expression of der}
The matrix expression of $\R \oplus \Der(\LG)$ with respect to the basis $\{x_1, \ldots, x_n\}$ is given by ${g_{\lambda, \xi}}^{-1}\Dd g_{\lambda, \xi}$, where
\begin{align*}
%\label{eq : matrix expression of R+der in lemma}
\Dd:=\left\{\left(
		\begin{array}{cc|ccc|c}
		\ast&\ast&0&\cdots&0&0\\ 
		\ast&\ast&0&\cdots&0&0\\ \hline
		\ast&\ast&\ast&\cdots&\ast&0\\
		\vdots&\vdots&\vdots&\ddots&\vdots&\vdots\\ 
		\ast&\ast&\ast&\cdots&\ast&0\\ \hline
		\ast&\ast&\ast&\cdots&\ast&\ast
		\end{array}
	\right) \in M(n, \R)\right\}.
\end{align*}
\end{Lem}

\begin{proof}
Recall that the matrix expression of $\R \oplus \Der(\LG)$ with respect to the standard basis $\{e_1, \ldots, e_n\}$ of $\LG$ coincides with $\Dd$ as we mentioned in Proposition~\ref{prop : aut}. Here we put
\begin{align*}
(x_1', \ldots, x_n')=(e_1, \ldots, e_n)g_{\lambda, \xi}.
\end{align*}
It follows from Lemma~\ref{lem : bracket of x'} that $\{x_1', \ldots, x_n'\}$ and $\{x_1, \ldots, x_n\}$ have the same bracket relation. Thus the matrix expressions of $\R \oplus \Der(\LG)$ with respect to these bases are the same. One can easily see that the matrix expression of $\R \oplus \Der(\LG)$ with respect to $\{x_1', \ldots, x_n'\}$ coincides with ${g_{\lambda, \xi}}^{-1}\Dd g_{\lambda, \xi}$, which completes the proof.
\end{proof}

%\textcolor{red}{This theorem refers only to the cases of the action of $\RAutg$ on $\GL(n, \R)/\OO(n)$. However note that it is also true for the cases of the action of $\RAutg$ on $\LLM_{(p, q)}(\LG)=\GL(n, \R)/\OO(p, q)$. (わざわざ書かなくてもよいか？)}

Recall that $\llu$ consists of six points, and parametrizes the orbit space. For each pair in $\llu$, the Ricci curvatures are calculated in Lemma~\ref{lem : ric}. One can then show the next theorem, which proves Theorem~\ref{thm : Ricci and Einstein}. 

\begin{Thm}
\label{prop : distinguished inner products}
Let $(\lambda, \xi) \in \llu$. Then $(\LG, \hoge< , >_{\lambda, \xi})$ is flat if and only if $(\lambda, \xi)=(1, 0)$. In the case of $(\lambda, \xi) \neq (1, 0)$, it satisfies that $(\LG, \hoge< , >_{\lambda, \xi})$ is an algebraic Ricci soliton but not Einstein. %\textcolor{red}{$($もう一度確認しておこうかな$)$}
\end{Thm}

\begin{proof}
Since $x_3, \ldots, x_{n-2}$ do not give any effect on the curvature tensor $R$, we have only to consider the case of $n=4$. Recall that 
\begin{align*}
(\lambda, \xi) \in \llu:=\{(0, 0), (1, 0), (1, 1), (2, 0), (2, \sqrt{3}), (2, 2)\}.
\end{align*}
First of all, assume that $(\LG, \hoge< , >_{\lambda, \xi})$ is Einstein, and prove $(\lambda, \xi)=(1, 0)$. We denote by $A \in M(4, \R)$ the matrix expression of $\Ric: \LG \to \LG$ with respect to $\{x_1, \ldots, x_4\}$. It follows from Lemma~\ref{lem : ric} that
\begin{align*}
2A=(\lambda^2-1)A_1+\xi^2 A_2,
\end{align*}
where the matrices $A_1$ and $A_2$ are defined by
\begin{align*}
A_1:=\left(
\begin{array}{cccc}
-\lambda^2-1&0&-\xi&-2\lambda\\
0&\ast&0&0\\
-\xi&0&-\xi^2&-\lambda\xi\\
2\lambda&0&\lambda\xi&\lambda^2+1
\end{array}
\right), \quad
A_2:=\left(
\begin{array}{cccc}
\lambda^2&0&0&\lambda\\
0&0&0&0\\
0&0&0&0\\
-\lambda&0&0&-1
\end{array}
\right). 
\end{align*}
Since $(\LG, \hoge< , >_{\lambda, \xi})$ is Einstein, the $(1, 3)$-component of $2A$ satisfies
\begin{align*}
-\xi(\lambda^2-1)=0.
\end{align*}
Hence one has $\lambda=1$ or $\xi=0$. If $\lambda=1$, then $2A=\xi^2 A_2$. Since $A_2$ is obviously not a scalar matrix, one has $\xi=0$. If $\xi=0$, then $2A=(\lambda^2-1)A_1$. One can see that $A_1$ is not a scalar matrix by comparing the $(1, 1)$ and $(4, 4)$-components, which yields $\lambda=1$. This shows that $(\lambda, \xi)=(1, 0)$.

Then the first assertion follows directly. If $(\LG, \hoge< , >_{\lambda, \xi})$ is flat, then it is Einstein, and hence $(\lambda, \xi)=(1, 0)$. Conversely,  if $(\lambda, \xi)=(1, 0)$, then one can directly show from Lemma~\ref{lem : curvature tensor} that the curvature tensor $R$ vanishes identically. 

Finally, we prove that $(\LG, \hoge< , >_{\lambda, \xi})$ is an algebraic Ricci soliton for any $(\lambda, \xi) \in \llu$. With respect to the basis $\{x_1, \ldots, x_4\}$, the matrix expression of $\Ric$ is $A$, and the matrix expression of $\R \oplus \Der(\LG)$ coincides with ${g_{\lambda, \xi}}^{-1}\Dd g_{\lambda, \xi}$ by Lemma~\ref{thm : matrix expression of der}. Therefore, in order to prove that $(\LG, \hoge< , >_{\lambda, \xi})$ is an algebraic Ricci soliton, we have only to show $A \in {g_{\lambda, \xi}}^{-1} \Dd g_{\lambda, \xi}$. For this purpose, it is enough to show that
\begin{align*}
%\label{criterion wrt A_1 and A_2}
g_{\lambda, \xi} A_1 {g_{\lambda, \xi}}^{-1}, g_{\lambda, \xi} A_2 {g_{\lambda, \xi}}^{-1} \in \Dd=\left\{\left(
\begin{array}{cc|c|c}
\ast&\ast&0&0\\
\ast&\ast&0&0\\ \hline
\ast&\ast&\ast&0\\ \hline
\ast&\ast&\ast&\ast
\end{array}
\right)\right\}.
\end{align*}
This can be proved by direct calculations in terms of
\begin{align*}
g_{\lambda, \xi}=\left(
\begin{array}{cccc}
1&&\xi&\lambda\\
&1&&\\
&&1&\\
&&&1
\end{array}
\right), \quad {g_{\lambda, \xi}}^{-1}=\left(
\begin{array}{cccc}
1&&-\xi&-\lambda\\
&1&&\\
&&1&\\
&&&1
\end{array}
\right).
\end{align*}
In fact, one can see that
\begin{align*}
g_{\lambda, \xi} A_1 {g_{\lambda, \xi}}^{-1}&=\left(
\begin{array}{cccc}
\lambda^2-\xi^2-1&0&0&0\\
0&\ast&0&0\\
-\xi&0&0&0\\
2\lambda&0&-\lambda\xi&-\lambda^2+1
\end{array}
\right),\\
g_{\lambda, \xi} A_2 {g_{\lambda, \xi}}^{-1}&=\left(
\begin{array}{cccc}
0&0&0&0\\
0&0&0&0\\
0&0&0&0\\
-\lambda&0&\lambda\xi&\lambda^2-1
\end{array}
\right).
\end{align*}
This completes the proof of the theorem.
\end{proof}

Recall that the flat inner product of signature $(n-1, 1)$ on $\LG$ corresponds to $(\lambda, \xi)=(1, 0)$, and it degenerates both on $Z(\LG)$ and $[\LG, \LG]$. In fact, a similar property holds for the cases of $G_{\RH^n}$ $(n \geq 2)$ and $H_3$ (\cite{KOTT, Nomizu, RR}). For $G_{\RH^n}$ with Lie algebra $\LG_{\RH^n}$, there is a unique flat inner product of an arbitrary signature up to scaling and automorphisms, and it degenerates on $[\LG_{\RH^n}, \LG_{\RH^n}]$ (note that $Z(\LG_{\RH^n})=0$). Similarly for $H_3$, there is a unique flat inner product of signature $(2, 1)$ on $\LH_3$ up to scaling and automorphisms, and it degenerates on $Z(\LH_3)=[\LH_3, \LH_3]$.

\begin{Rem}
\label{rem : isometric up to scaling on Lie group}
We here mention that there exist at least four left-invariant Lorentzian metrics on $H_3 \times \R^{n-3}$ $(n \geq 4)$ up to scaling and isometry. Recall that there exist exactly six such metrics up to scaling and automorphisms. First of all, we can distinguish only flat metric from the other Ricci soliton metrics up to scaling and isometry. For the other non-flat Ricci soliton metrics, we have calculated all the curvatures, and the best criterion would be the eigenvalues of the Ricci curvatures. According to Lemma~\ref{lem : bracket of x'}, the eigenvalues of the Ricci curvatures in the direction of $\{x_1, x_2, x_{n-1}, x_n\}$ can be calculated as below$:$ 
\begin{align*}
(0, 0) &: 1/2, 1/2, -1/2, 0,\\
(1, 1) &: 0, 0, 0, 0,\\
(2, 0) &: 9/2, 9/2, -9/2, 0,\\
(2, \sqrt{3}) &: 0, 0, 0, 0,\\
(2, 2) &: 3/2, -3/2, -3/2, 0.
\end{align*}
If two inner products are isometric up to scaling, then their sets of eigenvalues of the Ricci curvatures are the same up to $($positive$)$ scaling. Therefore, we have that
\begin{align*}
\hoge< , >_{1, 0}, \quad \hoge< , >_{0, 0}, \quad \hoge< , >_{1, 1}, \quad \hoge< , >_{2, 2}
\end{align*}
are all distinct up to scaling and isometry. However, as of now we are still not sure whether we can classify more finely. For example, we do not know whether two left-invariant Lorentzian metrics corresponding to $\hoge< , >_{0, 0}$ and $\hoge< , >_{2, 0}$, and also $\hoge< , >_{1, 1}$ and $\hoge< , >_{2, \sqrt{3}}$, respectively, are isometric up to scaling or not.
\end{Rem}

\section{Properties of the orbits} 
\label{sec6}

In this section, we study properties of the orbits of the action of $\RAutg$ on $\LLM_{(n-1, 1)}(\LG)$. In particular, we calculate the codimension of each orbit and determine all possible degenerations among the orbits, which yields our third main result. Remember $\LG:=\LH_3 \oplus \R^{n-3}$ with $n \geq 4$, and there exist exactly six orbits of the action of $\RAutg$ on $\LLM_{(n-1, 1)}(\LG)$. 

\subsection{The codimension of each orbit}

In this subsection, we calculate the codimensions of the six orbits of $\RAutg$ on $\LLM_{(n-1, 1)}(\LG)$. We use the inner product $\hoge< , >_{\lambda, \xi}$ for each $(\lambda, \xi) \in \llu$, defined in Lemma~\ref{lem : bracket of x'}. For simplicity of notation, we denote by 
\begin{align*}
H:=\RAutg, \quad [\hoge< , >_{\lambda, \xi}]:=H.\hoge< , >_{\lambda, \xi}.
\end{align*}
In order to calculate the codimensions of the orbits, we have only to know the dimensions of the stabilizers,
\begin{align*}
H_{\hoge< , >_{\lambda, \xi}}=H \cap \GL(n, \R)_{\hoge< , >_{\lambda, \xi}}=H \cap (g_{\lambda, \xi}\OO(n-1, 1){g_{\lambda, \xi}}^{-1}).
\end{align*}
First of all we study their Lie algebras. For this purpose, we take the block decomposition of size $(2, n-4, 2)$, and consider the subspaces
\begin{align*}
X:=\left\{\left(
\begin{array}{c|c|c}
\ast&&\ast\\ \hline
&&\\ \hline
\ast&&\ast
\end{array}
\right)\right\}, \quad
Y:=\left\{\left(
\begin{array}{c|c|c}
&&\\ \hline
&\ast&\\ \hline
&&
\end{array}
\right)\right\}, \quad
Z:=\left\{\left(
\begin{array}{c|c|c}
&\ast&\\ \hline
\ast&&\ast\\ \hline
&\ast&
\end{array}
\right)\right\}.
\end{align*}
Note that these subspaces are normalized by the conjugation of $g_{\lambda, \xi}$, namely, it satisfies $g_{\lambda, \xi}X{g_{\lambda, \xi}}^{-1}=X$, and so on.

\begin{Lem}
\label{lem : dim of three subspaces}
Let us put $\LH':=\LH \cap (g_{\lambda, \xi} \oo(n-1, 1) {g_{\lambda, \xi}}^{-1})$, where $\LH$ is the Lie algebra of $H$. Then we have
\begin{itemize}
\item[(1)]$\dim(X \cap \LH')=1+\dim U_{\lambda, \xi}$, where
\begin{align*}
U_{\lambda, \xi}:=
\left\{(b, d) \in \R^2 \middle| 
\begin{array}{l}
(\lambda^2-\xi^2-1)b=-\lambda\xi d\\
(\lambda^2-1)d=0
\end{array}
\right\},
\end{align*}
\item[(2)]$\dim(Y \cap \LH')=(1/2)(n-4)(n-5)$,
\item[(3)]$\dim(Z \cap \LH')=\dim W_{\lambda, \xi}$, where
\begin{align*}
W_{\lambda, \xi}:=\{(\bm{a}, \bm{c}) \in \R^{n-4} \times \R^{n-4} \mid 
(\lambda^2-1)\bm{a}=-\xi\bm{c}\}.
\end{align*}
\end{itemize}
\end{Lem}

\begin{proof}
Remember that $\Dd$ defined in Lemma~\ref{thm : matrix expression of der} is the matrix expression of $\LH$ with respect to the basis $\{e_1, \ldots, e_n\}$ of $\LG$.

First of all, we prove (1). For this subspace, we have only to consider the case of $n=4$ since four $(2 \times 2)$-corner blocks determine the dimension of $X \cap \LH'$. For any $x \in \oo(3, 1)$, one can write it as
\begin{align*}
x=\left(
\begin{array}{cc|cc}
0&a&b&d\\
-a&0&c&e\\ \hline
-b&-c&0&f\\
d&e&f&0\\
\end{array}
\right),
\end{align*}
where $a, b, c, d, e, f \in \R$. A direct calculation yields that
\begin{align*}
g_{\lambda, \xi} x {g_{\lambda, \xi}}^{-1}=
\left(
\begin{array}{cc|cc}
\ast&\ast&\xi^2b-\lambda\xi d+b+\lambda f&\lambda\xi b-\lambda^2d+d+\xi f\\
\ast&\ast&\xi a+c&\lambda a+e\\ \hline
\ast&\ast&\ast&\lambda b+f\\
\ast&\ast&\ast&\ast
\end{array}
\right).
\end{align*}
Hence $g_{\lambda, \xi} x {g_{\lambda, \xi}}^{-1} \in \Dd$ if and only if
\begin{align*}
%&\xi^2b-\lambda\xi d+b+\lambda f=0,\\
%&\lambda\xi b-\lambda^2d+d+\xi f=0,\\
&c=-\xi a, \quad e=-\lambda a, \quad f=-\lambda b, \\
&\xi^2b-\lambda\xi d+b-\lambda^2b=0, \quad -\lambda^2d+d=0.
\end{align*}
Therefore $g_{\lambda, \xi} x {g_{\lambda, \xi}}^{-1} \in \Dd$ is determined by $(a, b, d)$, where $a \in \R$ and $(b, d) \in U_{\lambda, \xi}$. This completes the proof of (1).

Next we prove (2). This case can be proved easily by
\begin{align*}
Y \cap \LH'=Y \cap (g_{\lambda, \xi} \oo(n-1, 1) {g_{\lambda, \xi}}^{-1}) \cong \oo(n-4).
\end{align*}

Finally we prove (3). Note that $Z=g_{\lambda, \xi} Z {g_{\lambda, \xi}}^{-1}$. Then we obtain
\begin{align*}
Z \cap (g_{\lambda, \xi} \oo(n-1, 1) {g_{\lambda, \xi}}^{-1})=g_{\lambda, \xi}(Z \cap \oo(n-1, 1)) {g_{\lambda, \xi}}^{-1}.
\end{align*}
For any $z \in Z \cap \oo(n-1, 1)$, one can write it as
\begin{align*}
z=\left(
\begin{array}{cc|c|cc}
&&\trans \bm{a}&&\\
&&\trans \bm{b}&&\\ \hline
-\bm{a}&-\bm{b}&&\bm{c}&\bm{d}\\ \hline
&&-\trans \bm{c}&&\\
&&\trans \bm{d}&&
\end{array}
\right),
\end{align*}
where $\bm{a}, \bm{b}, \bm{c}, \bm{d} \in \R^{n-4}$. Then a direct calculation yields that
\begin{align*}
g_{\lambda, \xi} z {g_{\lambda, \xi}}^{-1}=\left(
\begin{array}{cc|c|cc}
&&\trans \bm{a}-\xi \trans \bm{c}+\lambda \trans \bm{d}&&\\
&&\trans \bm{b}&&\\ \hline
-\bm{a}&-\bm{b}&&\xi \bm{a}+\bm{c}&\lambda \bm{a}+\bm{d}\\ \hline
&&-\trans \bm{c}&&\\
&&\trans \bm{d}&&
\end{array}
\right).
\end{align*}
Hence $g_{\lambda, \xi} z {g_{\lambda, \xi}}^{-1} \in \Dd$ if and only if
\begin{align*}
\bm{b}=0, \quad \bm{d}=-\lambda \bm{a}, \quad \bm{a}-\xi\bm{c}+\lambda\bm{d}=0.
\end{align*}
Therefore $g_{\lambda, \xi} z {g_{\lambda, \xi}}^{-1} \in \Dd$ is determined by $(\bm{a}, \bm{c}) \in W_{\lambda, \xi}$. This completes the proof.
\end{proof}

In the next proposition, we calculate the codimension of the orbit for each $(\lambda, \xi) \in \llu$. For the convenience, we also write the signatures of $\hoge< , >_{\lambda,\xi}$ restricted to $Z(\LG)$ and $[\LG, \LG]$ obtained in Proposition~\ref{prop : signature on subalgebra}.

\begin{Prop}
\label{prop : codim}
The codimension of  each orbit can be summarized as follows.
\begin{align*}
\begin{array}{|c||c|c|c|}\hline
\mbox{orbit}&\mbox{codimension}&\mbox{signature\ on}\ Z(\LG)&\mbox{signature\ on}\ [\LG, \LG]\\\hline\hline
[\hoge< , >_{0, 0}]&0&(n-3, 1, 0)&(0, 1, 0)\\\hline
[\hoge< , >_{1, 0}]&n-2&(n-3, 0, 1)&(0, 0, 1)\\\hline
[\hoge< , >_{1, 1}]&1&(n-3, 1, 0)&(0, 0, 1)\\\hline
[\hoge< , >_{2, 0}]&0&(n-2, 0, 0)&(1, 0, 0)\\\hline
[\hoge< , >_{2, \sqrt{3}}]&1&(n-3, 0, 1)&(1, 0, 0)\\\hline
[\hoge< , >_{2, 2}]&0&(n-3, 1, 0)&(1, 0, 0)\\\hline
\end{array}
\end{align*}
\begin{center}
$($signature convention $=(+, -, 0))$
\end{center}
\end{Prop}

\begin{proof}
By Lemma~\ref{lem : dim of three subspaces}, one obtains
\begin{align*}
\dim H_{\hoge< , >_{\lambda, \xi}}=\dim \LH'=1+(1/2)(n-4)(n-5)+\dim U_{\lambda, \xi}+\dim W_{\lambda, \xi}.
\end{align*}
Also one has 
\begin{align*}
&\dim \LLM_{(n-1, 1)}(\LG)=\dim (\GL(n, \R)/\OO(n-1, 1))=n(n+1)/2,\\
&\dim H=\dim \Dd=n^2-3n+7.
\end{align*}
Thus we have
\begin{align*}
\begin{split}
%\label{codim with H, lambda, xi}
\codim [\hoge< , >_{\lambda, \xi}]&=\dim \LLM_{(n-1, 1)}(\LG)-(\dim H-\dim H_{\hoge< , >_{\lambda, \xi}}) \\
&=\dim U_{\lambda, \xi}+\dim W_{\lambda, \xi}-(n-4).
\end{split}
\end{align*}
Hence we have only to calculate the dimensions of $U_{\lambda, \xi}$ and $W_{\lambda, \xi}$ for each $(\lambda, \xi) \in \llu$. First of all, we consider the case $(\lambda, \xi)=(0, 0)$. In this case, $U_{0, 0}$ and $W_{0, 0}$ are determined by
\begin{align*}
b=d=0, \quad \bm{a}=0, \quad \bm{c} \in \R^{n-4}.
\end{align*}
This shows $\dim U_{0, 0}=0$ and $\dim W_{0, 0}=n-4$, which yields that 
\begin{align*}
\codim [\hoge< , >_{0, 0}]=\dim U_{0, 0}+\dim W_{0, 0}-(n-4)=0.
\end{align*}
The remaining cases can be proved similarly. We here summarize the results:
\begin{align*}
(1, 0) &: b, d \in \R, \quad \bm{a}, \bm{c} \in \R^{n-4} ; \quad \dim U_{1, 0}+\dim W_{1, 0}=2n-6,\\
(1, 1) &: b=d, \quad \bm{a} \in \R^{n-4}, \quad \bm{c}=0; \quad \dim U_{1, 1}+\dim W_{1, 1}=n-3,\\
(2, 0) &: b=d=0, \quad \bm{a}=0, \quad \bm{c} \in \R^{n-4}; \quad \dim U_{2, 0}+\dim W_{2, 0}=n-4,\\
(2, \sqrt{3}) &: b \in \R, \quad d=0, \quad \bm{c}=-\sqrt{3}\bm{a}; \quad \dim U_{2, \sqrt{3}}+\dim W_{2, \sqrt{3}}=n-3,\\
(2, 2) &: b=d=0, \quad \bm{c}=-(3/2)\bm{a}; \quad \dim U_{2, 2}+\dim W_{2, 2}=n-4.
\end{align*}
This completes the proof.
\end{proof}

\subsection{The degeneration of each orbit}

In this subsection, we study degenerations of orbits of the action of $\RAutg$ on $\LLM_{(n-1, 1)}(\LG)$, and prove Theorem~\ref{thm : flat and closed}. First of all we recall the definition of degenerations.

\begin{Def}
Let $\Oo_1$ and $\Oo_2$ be orbits of some action, and assume that $\Oo_1 \neq \Oo_2$. Then, $\Oo_1$ is said to {\it degenerate} to $\Oo_2$ if $\Oo_2 \subset \overline{\Oo_1}$ holds, where $\overline{\Oo_1}$ is the closure of $\Oo_1$.
\end{Def}

In this paper, we denote $\Oo_1 \to \Oo_2$ when $\Oo_1$ degenerates to $\Oo_2$. Recall that there are exactly six orbits of the action of $\RAutg$ on $\LLM_{(n-1, 1)}(\LG)$.

\begin{Prop}
All possible degenerations among the six orbits of the action of $\RAutg$ on $\LLM_{(n-1, 1)}(\LG)$ are given as follows$:$
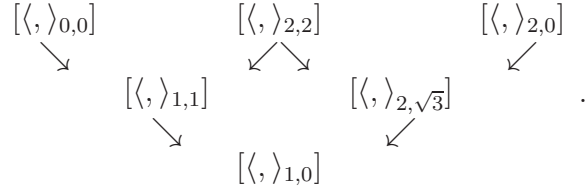
\begin{align*}
%\label{diagram of degeneration}
\begin{array}{ccccc}
	[\hoge< , >_{0, 0}]&&[\hoge< , >_{2, 2}]&&[\hoge< , >_{2, 0}]\\ %\hline
	\searrow&&\swarrow \searrow&&\swarrow\\ %\hline
	&[\hoge< , >_{1, 1}]&&[\hoge< , >_{2, \sqrt{3}}]&\\ %\hline
	&\searrow&&\swarrow&\\ %\hline
	&&[\hoge< , >_{1, 0}]&&\\ 
\end{array}.
\end{align*}
\end{Prop}

\begin{proof}
We prove that
\begin{itemize}
\item[(1)]the six degenerations in the diagram do occur, and
\item[(2)]other degenerations do not occur.
\end{itemize}
First of all we prove (1). Note that, in order to prove $\Oo_1 \to \Oo_2$, we have only to show that one point in $\Oo_2$ is contained in $\overline{\Oo_1}$. In fact, if there exists $p \in \Oo_2$ such that $p \in \overline{\Oo_1}$, then we have $g.p \in \overline{\Oo_1}$ for any $g \in \RAutg$, which means $\Oo_2 \subset \overline{\Oo_1}$.

In order to show the degenerations, we consider inner products $\hoge< , >_{\lambda, \xi}$ on $\LG$ defined in Lemma~\ref{lem : bracket of x'}, where $\lambda, \xi \in \R$. Recall that there exists a pseudo-orthonormal basis $\{x_1, \ldots, x_n\}$ with respect to $\hoge< , >_{\lambda, \xi}$ whose bracket relation is given by
\begin{align*}
[x_1, x_2]=-(\lambda x_1-x_n), \quad [x_2, x_{n-1}]=\xi(\lambda x_1-x_n), \quad [x_2, x_n]=\lambda(\lambda x_1-x_n).
\end{align*}
Note that, according to Proposition~\ref{prop : signature on subalgebra}, the orbit $[\hoge< , >_{\lambda, \xi}]$ is determined by the signatures of the restrictions of $\hoge< , >_{\lambda, \xi}$ to $Z(\LG)$ and $[\LG, \LG]$. One knows
\begin{align*}
[\LG, \LG]=\s\{\lambda x_1-x_n\}, \quad \hoge<\lambda x_1-x_n, \lambda x_1-x_n>_{\lambda, \xi}=\lambda^2-1.
\end{align*}
For the center $Z(\LG)$, we have a basis $\{y_1, \ldots, y_{n-2}\}$ of $Z(\LG)$ given in Lemma~\ref{lem : basis of center}, which is orthogonal with respect to $\hoge< , >_{\lambda, \xi}$, and $\hoge<y_i, y_i>_{\lambda, \xi}=1$ holds for every $i \in \{1, \ldots, n-3\}$. Therefore, the signature of $\hoge< , >_{\lambda, \xi}$ on $Z(\LG)$ can be determined by
\begin{align*}
y_{n-2}:=\lambda(x_1+\xi x_{n-1})-(\xi^2+1)x_n, \quad \hoge<y_{n-2}, y_{n-2}>_{\lambda, \xi}=(\xi^2+1)(\lambda^2-\xi^2-1).
\end{align*}

Using these facts, first of all we show $[\hoge< , >_{0, 0}] \to [\hoge< , >_{1, 1}]$. Consider a family of inner products $\hoge< , >_{t, t}$ with $t \in [0, 1)$. Then the signature of $\hoge< , >_{t, t}$ is $(0, 1, 0)$ on $[\LG, \LG]$, since
\begin{align*}
\hoge<t x_1-x_n, t x_1-x_n>_{t, t}=t^2-1<0.
\end{align*}
Similarly the signature of $\hoge< , >_{t, t}$ is $(n-3, 1, 0)$ on $Z(\LG)$, since
\begin{align*}
\hoge<y_{n-2}, y_{n-2}>_{t, t}=(t^2+1)(t^2-t^2-1)<0.
\end{align*}
This yields that
\begin{align*}
\hoge< , >_{t, t} \in [\hoge< , >_{0, 0}] \quad (t \in [0, 1)).
\end{align*}
Therefore, by taking the limit under $t \to 1$, this shows
\begin{align*}
\hoge< , >_{1, 1} \in \overline{[\hoge< , >_{0, 0}]},
\end{align*}
which completes the proof of $[\hoge< , >_{0, 0}] \to [\hoge< , >_{1, 1}]$. We can similarly prove the other degenerations. Here we summarize families of inner products to show each degeneration:
\begin{align*}
[\hoge< , >_{1, 1}] \to [\hoge< , >_{1, 0}]& : \hoge< , >_{1, t} \in [\hoge< , >_{1, 1}] \quad (t \in (0, 1]),\\
[\hoge< , >_{2, 0}] \to [\hoge< , >_{2, \sqrt{3}}]& : \hoge< , >_{2, t} \in [\hoge< , >_{2, 0}] \quad (t \in [0, \sqrt{3})),\\
[\hoge< , >_{2, \sqrt{3}}] \to [\hoge< , >_{1, 0}]& : \hoge< , >_{s, t} \in [\hoge< , >_{2, \sqrt{3}}] \quad (s \in (1, 2],\ t \in (0, \sqrt{3}],\ s^2=t^2+1),\\
[\hoge< , >_{2, 2}] \to [\hoge< , >_{1, 1}]& : \hoge< , >_{t, t} \in [\hoge< , >_{2, 2}] \quad (t \in (1, 2]),\\
[\hoge< , >_{2, 2}] \to [\hoge< , >_{2, \sqrt{3}}]& : \hoge< , >_{2, t} \in [\hoge< , >_{2, 2}] \quad (t \in (\sqrt{3}, 2]).
\end{align*}

We next prove (2). By a general theory, if an orbit $\Oo_1$ degenerates to another orbit $\Oo_2$, then one has
\begin{align*}
\dim \Oo_1 > \dim \Oo_2.
\end{align*}
One knows the dimensions of all orbits in Proposition~\ref{prop : codim}. This yields that, in the diagram of the assertion, horizontal arrows and upward arrows do not occur. It remains to show that the following two degenerations do not occur:
\begin{align*}
[\hoge< , >_{0, 0}] \to [\hoge< , >_{2, \sqrt{3}}], \quad [\hoge< , >_{2, 0}] \to [\hoge< , >_{1, 1}].
\end{align*}
This follows from the signatures of inner products. For example, the signatures of $\hoge< , >_{0, 0}$ and $\hoge< , >_{2, \sqrt{3}}$ on $[\LG, \LG]$ are $(+, -, 0)=(0, 1, 0)$ and $(1, 0, 0)$, respectively. Hence, if $[\hoge< , >_{0, 0}]$ degenerates to $[\hoge< , >_{2, \sqrt{3}}]$, then a negative eigenvalue converges to a positive eigenvalue, with skipping $0$. This is a contradiction. By a similar argument, one can show that $[\hoge< , >_{2, 0}]$ does not degenerate to $[\hoge< , >_{1, 1}]$, which completes the proof of (2).
\end{proof}

In general, an orbit is closed if and only if it does not degenerate to any other orbit. Hence, by this proposition, $[\hoge< , >_{1, 0}]$ is a unique closed orbit in $\LLM_{(n-1, 1)}(\LG)=\GL(n, \R)/\OO(n-1, 1)$. Note that it is the unique equivalence class of flat left-invariant Lorentzian metrics on $H_3 \times \R^{n-3}$ $(n \geq 4)$ up to scaling and automorphisms from Proposition~\ref{prop : distinguished inner products}. Therefore we obtain Theorem~\ref{thm : flat and closed}.

\end{document}